\documentclass[12pt,reqno]{amsart}
\usepackage{hyperref}
\usepackage{amsmath,amssymb}
\usepackage[english]{babel}
\usepackage{caption}
\usepackage{amssymb,amsmath,euscript,enumerate,tikz}
\usepackage[margin=1in]{geometry}
\usepackage{graphicx}
\usepackage{float}
\usepackage{pgf,tikz}
\usepackage{mathrsfs}
\usepackage{upgreek}

\newcommand \reg{\operatorname{reg}}

\newcommand \Tor{\operatorname{Tor}}

\newcommand \pd{\operatorname{pd}}
\newcommand \iv{\operatorname{iv}}

\newcommand \ass{\operatorname{Ass}}

\newcommand \depth{\operatorname{depth}}

\newcommand \K{\mathbb{K}}
\newcommand{\T}{\operatorname{Tor}}
\newcommand{\ext}{\operatorname{Ext}}

\newtheorem{theorem}{Theorem}[section]
\newtheorem{definition}[theorem]{Definition}
\newtheorem{lemma}[theorem]{Lemma}
\newtheorem{proposition}[theorem]{Proposition}

\newtheorem{question}[theorem]{Question}

\newtheorem{corollary}[theorem]{Corollary}

\newtheorem{notation}[theorem]{Notation}

\newtheorem{remark}[theorem]{Remark}

\begin{document}
\title[Binomial Edge Ideals of Unicyclic Graphs]{Binomial Edge Ideals of Unicyclic Graphs}
\author[Rajib Sarkar]{Rajib Sarkar}
\email{rajib.sarkar63@gmail.com}
\address{Department of Mathematics, Indian Institute of Technology Madras, \newline \indent Chennai - 600036, India}

\begin{abstract}
Let $G$ be a connected graph on the vertex set $[n]$. Then $\depth(S/J_G)\leq n+1$. In this article, we prove that if $G$ is a unicyclic graph, then the depth of $S/J_G$ is bounded below by $n$. Also, we characterize $G$ with $\depth(S/J_G)=n$ and $\depth(S/J_G)=n+1$. We then compute one of the distinguished extremal Betti numbers of $S/J_G$. If $G$ is obtained by attaching whiskers at some vertices of the cycle of length $k$, then we show that $k-1\leq \reg(S/J_G)\leq k+1$. Furthermore, we characterize $G$ with $\reg(S/J_G)=k-1$, $\reg(S/J_G)=k$ and $\reg(S/J_G)=k+1$. In each of these cases, we classify the uniqueness of the extremal Betti number of these graphs.
\end{abstract}
\keywords{Binomial edge ideal, Depth, Extremal Betti number, Castelnuovo-Mumford regularity}
\thanks{AMS Subject Classification (2010): 13D02,13C13, 05E40}
\maketitle
\section{Introduction}
Let $R = \K[x_1,\ldots,x_m]$ be the standard graded polynomial ring over an arbitrary
field $\K$ and $M$ be a finitely generated graded  $R$-module. 
Let
\[
0 \longrightarrow \bigoplus_{j \in \mathbb{Z}} R(-p-j)^{\beta_{p,p+j}(M)} 
{\longrightarrow} \cdots {\longrightarrow} \bigoplus_{j \in \mathbb{Z}} R(-j)^{\beta_{0,j}(M)} 
{\longrightarrow} M\longrightarrow 0,
\]
be the minimal graded free resolution of $M$. The number $\beta_{i,j}(M)$ is called the 
$(i,j)$-th \textit{graded Betti number} of $M$. From the minimal free resolution of a graded module, one can obtain two important invariants, namely the projective dimension and the Castelnuovo-Mumford regularity. The 
\textit{projective dimension} of $M$, denoted by $\pd(M)$, is defined as \[\pd(M):=\max\{i : \beta_{i,i+j}(M) \neq 0 \text{ for some } j\}\]
and the \textit{Castelnuovo-Mumford regularity} (or simply, \textit{regularity}) of $M$, denoted by $\reg(M)$, is defined as 
\[\reg(M):=\max \{j : \beta_{i,i+j}(M) \neq 0 \text{ for some } i\}. \]
If $\beta_{i,i+j}(M)\neq 0$ and for all pairs $(k,l)\neq (i,j)$ with $k\geq i$ and $l\geq j$, $\beta_{k,k+l}(M)=0$, then $\beta_{i,i+j}(M)$ is called an \textit{extremal Betti number} of $M$. If $p =\pd(M)$ and $r=\reg(M)$, then there exist unique numbers $i$ and $j$ such that $\beta_{p,p+i}(M)$ and $\beta_{j,j+r}(M)$ are extremal Betti numbers. These extremal Betti numbers are called the \textit{distinguished extremal Betti numbers} of $M$, see \cite{her2}. Note that $M$ admits a unique extremal Betti number if and only if 
$\beta_{p,p+r}(M)\neq 0$.

Let $G$ be a  simple graph on the vertex set $V(G) =[n]:= \{1, \ldots, n\}$ and the edge set $E(G)$. Let $S=
\K[x_1, \ldots, x_{n}, y_1, \ldots, y_{n}]$ be the polynomial ring on $2n$ variables over an arbitrary field $\K$. The \textit{binomial edge ideal} of $G$, denoted by $J_G$, defined as $J_G=(
x_i y_j - x_j y_i ~ : i < j \text{ and } \{i,j\}\in E(G)) \subseteq S$ 
was introduced by Herzog et al. in \cite{HH1} and independently by 
Ohtani in \cite{oh}. In the recent past, 
there has been considerable interest in computing algebraic invariants such as depth and regularity of $J_G$ in terms of combinatorial invariants such as clique number, number of vertices, length of a longest induced path and number of internal vertices of $G$, see \cite{Montaner,Arindam,dav,ERT20,KM-JCTA,Rinaldo-JAA,MM,MKM-Conj,MKM-Chordal,KM-EJC} for a partial list.

It is well known (see, for example \cite[Proposition 1.2.13]{bh}) that $\depth(S/J_G)\leq \dim(S/P)$ for all $P\in \ass(J_G)$. It follows from \cite[Theorem 3.2]{HH1} that $P=P_{\emptyset}(G)\in \ass(J_G)$ with $\dim(S/P)=n+1$. Therefore, $\depth(S/J_G)\leq n+1$ for every connected graph $G$ on $n$ vertices. In general, there is no lower bound for the depth of $S/J_G$. In \cite[Theorem 4.5]{ZafarBMS}, Zafar proved that $\depth(S/J_G)=n$ where $G$ is a cycle on $n$ vertices. If $G=G_1*G_2$, the join product of $G_1$ and $G_2$, then Kumar and the present author gave a formula for the depth of $S/J_G$ in terms of the depths of $S_{G_1}/J_{G_1}$ and $S_{G_2}/J_{G_2}$, see \cite[Theorems 4.1, 4.3 and 4.4]{AR2}. Recently, Rouzbahani Malayeri, Saeedi Madani and Kiani studied the depth of $S/J_G$ and they characterized all graphs $G$ such that $\depth(S/J_G)=4$ in \cite{MKM-JA}. Let $G$ be a connected unicyclic graph of girth $k$ on $n$ vertices with $n> k$ for $k\geq 3$. If $k=3$, then $G$ is a chordal graph with the property that any two maximal cliques intersect in at most one vertex. In \cite{her1}, Ene, Herzog and Hibi proved that $\depth(S/J_G)=n+1$ for such graphs. For $k\geq 4$, we compute the depth of $S/J_G$ in a slightly more general setting.
\vskip 2mm 
\noindent
\textbf{Theorem \ref{depth-girth}. }{\em 
Let $k\geq 3$ and $m\geq 2$. Let $G$ be the clique sum of $H=C_k\cup_e K_m$ and a forest along some vertices of $H$. Then $\depth(S/J_G)\geq n$. Let $A=\{u\in V(C_k): \text{there is a tree incident on }u\}$. If $e\cap A \neq \emptyset$ and $G[A]$ is connected with $k-2\leq |A|$, then $\depth(S/J_G)=n+1$.}
\vskip 2mm 
 Considering $m=2$, we obtain our results for unicyclic graph. Moreover, we prove that if there are trees attached to $k-2$ consecutive vertices of the cycle in $G$, then $\depth(S/J_G)=n+1$ and otherwise, $\depth(S/J_G)=n$, see Corollary \ref{depth-unicyclic}. 

In \cite{MM}, Matsuda and Murai proved that $\ell(G)\leq \reg(S/J_G)\leq n-1$, where $\ell(G)$ is the length of a longest induced path in $G$. It is evident that $\ell(G)$ is not a sharp lower bound. If $G$ is assumed to be a tree, then Chaudhry et al. \cite{CDI16} proved that $\reg(S/J_G)=\ell(G)$ if and only if $G$ is a caterpillar. An improved lower bound for trees was obtained by Jayanthan et al. in \cite{JNR1}, where they proved that $\iv(G)+1\leq \reg(S/J_G)$. In the case of $G$ being a block graph, Herzog and Rinaldo \cite{her2} generalized their result and proved that $\iv(G)+1\leq \reg(S/J_G)$ and they also characterized $G$ admitting a unique extremal Betti number.
There have been some other works as well on the computation of extremal Betti numbers of binomial edge ideals. In \cite{Alba-Hoang}, de Alba and Hoang studied extremal Betti numbers of binomial edge ideals of closed graphs, and Kumar studied extremal Betti numbers of binomial edge ideals of generalized block graphs, \cite{Arv-GBG}. Recently, Mascia and Rinaldo \cite{Rinaldo-extremal} studied extremal Betti numbers of some Cohen-Macaulay bipartite graphs. In this article, we study extremal Betti numbers of $J_G$, and as a consequence, we obtain a lower bound for the regularity of $J_G$, where $G$ is a unicyclic graph.

\vskip 2mm 
\noindent
\textbf{Corollary \ref{betti-unicyclic}.} {\em Let $G$ be a unicyclic graph of girth $k\geq 4$ with $\pd(S/J_G)=p$. If trees are attached to every vertex of the cycle in $G$, then $\beta_{p,p+\iv(G)+1}(S/J_G)$ is an extremal Betti number of $S/J_G$, and hence $\iv(G)+1\leq \reg(S/J_G)$. Otherwise, either $\beta_{p,p+\iv(G)-1}(S/J_G)$ or $\beta_{p,p+\iv(G)}(S/J_G)$ is an extremal Betti number of $S/J_G$, and hence $\iv(G)-1\leq \reg(S/J_G)$.}

\vskip 2mm 
There are only few classes of graphs for which the regularity of their binomial edge ideals are known, see \cite{EZ,JA1,KM-JA,Schenzel,Zafar}.
In the last section, we study the regularity and behavior of extremal Betti number of graphs $G$, where $G$ is obtained by attaching whiskers to some vertices of the cycle of length $k$. We first prove that the regularity of $S/J_G$ is bounded below by $k-1$ and bounded above by $k+1$. We then characterize $G$ such that $\reg(S/J_G)=k+1$, $\reg(S/J_G)=k-1$ and $\reg(S/J_G)=k$.
 \vskip 2mm
 \noindent
  \textbf{Corollary \ref{reg-corollary}.} {\em Let $G=C_k\cup (\cup_{i=1}^{k} W^{r_i}(v_i))$ for $r_i\geq 0$ and $k\geq 4$. Let $A=\{v_i\in V(C_k): r_i\geq 1\}$ and suppose that $A\neq \emptyset$. Then $k-1\leq \reg(S/J_G)\leq k+1$. Moreover,
  \begin{enumerate}[(1)]
  	\item $\reg(S/J_G)=k+1$ if and only if $A=V(C_k)$,
  	\item $\reg(S/J_G)=k-1$ if and only if $|A|=1$ or $|A|=2$ and vertices of $A$ are adjacent,
  	\item $\reg(S/J_G)=k$ if and only if $A$ contains at least two non-adjacent vertices and $A\subsetneq V(C_k)$.
  \end{enumerate}
}
Furthermore, we show that if $\reg(S/J_G)=k+1$ and $\reg(S/J_G)=k-1$, then $S/J_G$ admits a unique extremal Betti number. If $\reg(S/J_G)=k$, then $S/J_G$ does not always admit a unique extremal Betti number. In this case, we classify $G$ such that $S/J_G$ admits a unique extremal Betti number.

\vskip 2mm
\noindent
\textbf{Acknowledgment:} The author is grateful to his Ph.D. advisor, Prof. A. V. Jayanthan for his
constant support and insightful discussions. The author thanks 
University Grants Commission, Government of India, for financial assistance. The author has extensively used computer algebra software Macaulay2 \cite{m2} for computations. The author is also thankful to the referee for carefully reading the manuscript and making several suggestions that improved the exposition.

\section{Preliminaries}
 
Let $G$  be a  simple graph with the vertex set $[n]$ and edge set
$E(G)$. A graph $G$ on the vertex set $[n]$ is said to be a \textit{complete graph}, if
$\{i,j\} \in E(G)$ for all $1 \leq i < j \leq n$. We denote the complete graph on $n$ vertices by $K_n$. For $A \subseteq V(G)$, the
\textit{induced subgraph} of $G$ on the vertex set $A$, denoted by $G[A]$, is the graph such that for
$i, j \in A$, $\{i,j\} \in E(G[A])$ if and only if $ \{i,j\} \in
E(G)$.  For a vertex $v\in V(G)$, let $G \setminus v$ denote the  induced
subgraph of $G$ on the vertex set $V(G) \setminus \{v\}$. For a subset
$U\subseteq V(G)$, if the induced subgraph $G[U]$ is a complete graph then $U$ is called a \textit{clique}. A vertex $v$ of $G$ is said to be a \textit{simplicial vertex}
if $v$ belongs to only one maximal clique. If $v$ is not a simplicial vertex, then $v$ is called an \textit{internal vertex}. The number of internal vertices of $G$ is denoted by $\iv(G)$. For a vertex $v$ in $G$,
$N_G(v) := \{u \in V(G) :  \{u,v\} \in E(G)\}$ denotes the
\textit{neighborhood} of $v$ in $G$ and  $G_v$ is the graph with the vertex set
$V(G)$ and edge set $E(G_v) =E(G) \cup \{ \{u,w\}: u,w \in N_G(v)\}$ i.e., $G_v$ is obtained from $G$ by making a complete graph on $N_G(v)\cup \{v\}$ in $G$. Set $N_G[v]:=N_G(v)\cup \{v\}$.
The \textit{degree} of a vertex  $v$, denoted by $\deg_G(v)$, is
$|N_G(v)|$. A
\textit{cycle} on the vertex set $[k]$, denoted by $C_k$, is a graph with the edge set $\{i,i+1:1\leq i\leq k-1\}\cup \{1,k\}$ for $k\geq 3$. A
graph is said to be a \textit{unicyclic} graph if it contains exactly
one cycle as a subgraph. The \textit{girth} of a graph $G$ is the length of a shortest
cycle in $G$. A graph $G$ is called \textit{chordal} if every induced cycle of $G$ has $3$ vertices. A connected graph is a \textit{tree} if it does not have a cycle. A \textit{forest} is a disconnected graph whose components are trees. A vertex $v\in V(G)$ is said to be a \textit{cut vertex} if $G\setminus v$ has more components than $G$. A \textit{block} of a graph is a maximal nontrivial connected subgraph which has no cut vertex. 
A graph $G$ is called a \textit{block graph} if every block of $G$ is a complete graph. It is easy to see that $G$ is a block graph if and only if $G$ is a chordal graph with the property that any two maximal cliques intersect in at most one vertex.  A connected chordal graph $G$ is said to be a \textit{generalized block graph} if three maximal cliques of $G$ intersect non-trivially, then the intersection of each pair of them is the same.

For $T \subseteq V(G)$, let $c(T)$
denote the number of components of $G[\bar{T}]$, where $\bar{T} = V(G)\setminus T$. Also, let $G_1,\cdots,G_{c(T)}$ be the 
components of $G[\bar{T}]$ and for every $i$, $\tilde{G_i}$ denotes the
complete graph on the vertex set $V(G_i)$. Moreover, we set $P_T(G) := (\underset{i\in T} \cup \{x_i,y_i\}, J_{\tilde{G_1}},\cdots, J_{\tilde{G}_{c(T)}})$.
In \cite{HH1}, Herzog et al. proved that $J_G =  \underset{T \subseteq [n]}\cap
P_T(G)$, which in particular, implies that $J_G$ is a radical ideal. A set $T\subseteq V(G)$ is said to have \textit{cut point property} if for every
$i \in T$, $i$ is a cut vertex of the graph $G[\bar{T} \cup \{i\}]$ i.e., $c(T\setminus \{i\} )< c(T)$.
They also showed 
that $P_T(G)$ is a minimal prime
of $J_G$ if and only if either $T = \emptyset$ or $T \subset V(G)$ has the cut point property, see \cite[Corollary 3.9]{HH1}.

\vskip 2mm

Let $R=\K[x_1,\dots,x_m]$, $R'=\K[x_{m+1},\dots,x_n]$ and $R''=\K[x_1,\dots,x_n]$ be
polynomial rings. Let $I\subseteq R$ and $J\subseteq R'$ be homogeneous ideals. Then it is well known that the minimal free resolution of $R''/(I+J)$ is the tensor product of the minimal free resolutions of $R/I$ and $R'/J$. Therefore, we have for all $i,j$, 
\begin{align}\label{betti-product}
\beta_{i,i+j}\left(\frac{R''}{I+J}\right)= \underset{{\substack{i_1+i_2=i \\ j_1+j_2=j}}}{\sum}\beta_{i_1,i_1+j_1}\left(\frac{R}{I}\right)\beta_{i_2,i_2+j_2}\left(\frac{R'}{J}\right).
\end{align}
The following depth lemma and regularity lemma can be easily derived from the long exact sequence of $\T$ and $\ext $ corresponding to given short exact sequence.
\begin{lemma}\label{depth-lemma}
	Let $R$ be a standard graded ring  and $M,N,P$ be finitely generated graded $R$-modules. 
	If $ 0 \rightarrow M \xrightarrow{f}  N \xrightarrow{g} P \rightarrow 0$ is a 
	short exact sequence with $f,g$  
	graded homomorphisms of degree zero, then 
	\begin{enumerate}[(1)]
		\item $\depth(M) \ge \min \{\depth(N), \depth(P)+1\},$
		\item $\depth(N) \ge \min \{\depth(M), \depth(P)\},$
		\item $\depth(P) \ge \min \{\depth(M)-1, \depth(N)\},$ 
		\item $\depth(M) = \depth(P)+1,$ {\em if}  $\depth(N) > \depth(P)$ {\em and} 
		\item $\depth(M) = \depth(N),$ {\em if }  $\depth(N) < \depth(P)$.
	\end{enumerate}	
\end{lemma}
\begin{lemma}\label{regularity-lemma}
	Let $R$ be a standard graded ring  and $M,N,P$ be finitely generated graded $R$-modules. 
	If $ 0 \rightarrow M \xrightarrow{f}  N \xrightarrow{g} P \rightarrow 0$ is a 
	short exact sequence with $f,g$  
	graded homomorphisms of degree zero, then 
	\begin{enumerate}[(1)]
		\item $\reg(M)\leq \max \{\reg(N),\reg(P)+1 \}$,
		\item $\reg (N) \leq \max \{\reg (M), \reg (P)\}$,
		\item $\reg(P)\leq \max \{\reg(M)-1,\reg(N)\}$,
		\item $\reg (M) = \reg (P)+1,$ {\em if} $\reg (N) < \reg (M)$ {\em and}
		\item $\reg(M)=\reg(N),$ {\em if} $\reg(N)> \reg(P)$.
	\end{enumerate}	
\end{lemma}	
The following is a crucial lemma due to Ohtani, which is used repeatedly throughout this article.
\begin{lemma}$($\cite[Lemma 4.8]{oh}$)$\label{ohtani-lemma}
	Let $G$ be a  graph on $V(G)$ and $v\in V(G)$ such that $v$ is not a simplicial vertex. Then $J_{G}=(J_{G\setminus v}+(x_v,y_v))\cap J_{G_v}$.
\end{lemma}
Thus, we get the following short exact sequence:
      \begin{align}\label{ohtani-ses}
    0\longrightarrow \frac{S}{J_{G}}\longrightarrow 
    \frac{S}{(x_v,y_v)+J_{G \setminus v}} \oplus \frac{S}{J_{{G_v}}}\longrightarrow \frac{S}{(x_v,y_v)+J_{G_v \setminus v}} \longrightarrow 0,
     \end{align}
  and correspondingly the  long exact sequence of Tor modules:
       \begin{multline}\label{ohtani-tor}
\cdots \longrightarrow \Tor_{i}^{S}\left( \frac{S}{J_G},\K\right)_{i+j}\longrightarrow \Tor_{i}^{S}\left( \frac{S}{(x_v,y_v)+J_{G\setminus v}},\K\right)_{i+j} 
\oplus \Tor_{i}^{S}\left(\frac{S}{J_{G_v}},\K\right)_{i+j} \\ \longrightarrow \Tor_{i}^{S}\left(\frac{S}{(x_v,y_v)+J_{G_v\setminus v}},\K\right)_{i+j} \longrightarrow \Tor_{i-1}^{S}\left( \frac{S}{J_G},\K\right)_{i+j}\longrightarrow \cdots
     \end{multline}

\section{Unicyclic Graph}
Let $G$ be a connected unicyclic graph (which is not a cycle) of girth $k$ on $n$ vertices for $k\geq 4$. In this section, we prove that $n\leq \depth(S/J_G)$ and we characterize unicyclic graphs $G$ such that $\depth(S/J_G)=n$. We also compute one distinguished extremal Betti number of $S/J_G$.
\begin{notation}{\em
		Let $G$ be a graph on $V(G)=[n]$. We reserve the notation $S$ for the
		polynomial ring $\K[x_i, y_i : i \in [n]]$ and for $v\in V(G)$, $S'$ for the polynomial ring $\K[x_i,y_i : i\in V(G)\setminus \{v\}]$.  If $H$ is any other graph with the vertex set $V(H)$, then we set
		$S_H = \K[x_i, y_i : i \in V(H)]$ and for $v\in V(H)$, set $S_H'=\K[x_i,y_i : i\in V(H)\setminus \{v\}]$.
		}
\end{notation}
To re-emphasize: Unless stated otherwise, $G$ always denotes a graph on $n$ vertices.
\begin{definition}
	Let $G_1$ and $G_2$ be two subgraphs of a graph $G$. If $G_1\cap
	G_2=K_m$, $V(G_1)\cup V(G_2)=V(G)$ and $E(G_1)\cup E(G_2)=E(G)$ with $G_1\neq K_m$ then $G$ is called the clique sum of $G_1$ and $G_2$
	along the complete graph $K_m$, denoted by $G_1\cup_{K_m} G_2$. Sometimes we call this clique sum that $G_1$ is attached to $G_2$ along $K_m$. If $G_2=K_m$, then $G_1\cup_{K_m} G_2=G_1$.
\end{definition}
\begin{notation}
	Let $k\geq 3$. For the rest of the article, we fix the following notation for the cycle graph on $k$-vertices. Let $V(C_k)=\{v=v_1,v_2,\dots,v_k\}$ be such that $E(C_k)=\{\{v_i,v_{i+1}\},\{v_1,v_k\}:1\leq i\leq k-1 \}$.
\end{notation}
 Let $H$ be the clique sum of $C_k$ and a complete graph $K_m$ along an edge $e$ for $m\geq 2$. If $m=2$, then $H=C_k$ and in this case, Zafar and Zahid \cite[Corollary 16]{Zafar} proved that $\reg(S_H/J_H)=k-2$ and $\beta_{k,2k-2}(S_H/J_H)$ is the unique extremal Betti number of $S_H/J_H$. For $m\geq 3$, Jayanthan et al. proved that $\reg(S_H/J_H)=k-1$, see \cite[Proposition 3.11]{JAR2}. Here, we prove that $S_H/J_H$ admits a unique extremal Betti number, namely $\beta_{n,n+k-1}(S_H/J_H)$, where $|V(H)|=n$.
\begin{proposition}\label{betti-cliquesum}
	Let $k,m\geq 3$ and $H=C_k\cup_e K_m$, with $|V(H)|=n$, be the clique sum of a cycle $C_k$ and a complete graph $K_m$ along an edge $e$. Then $\depth(S_H/J_H)=n$ and $\beta_{n,n+k-1}(S_H/J_H)$ is the unique extremal Betti number of $S_H/J_H$.
\end{proposition}
\begin{proof}
	Let $e=\{v,v_2\}$. We proceed by induction on $k$. Suppose first that $k=3$. Then $H=C_3\cup_e K_m$, which is a generalized block graph. Thus by \cite[Theorem 3.2]{KM-CA}, $\depth(S_H/J_H)=n$, and hence it follows from \cite[Theorem 3.7]{Arv-GBG} that $\beta_{n,n+2}(S_H/J_H)$ is the unique extremal Betti number of $S_H/J_H$. 
	
	Now suppose that $k\geq 4$. By Lemma \ref{ohtani-lemma}, $J_H=J_{H_v}\cap ((x_v,y_v)+J_{H\setminus v})$, where $H_v=C_{k-1}\cup_{e'} K_{m+1}$ and $H\setminus v=P_{k-1}\cup_{v_2} K_{m-1}$ with $e'=\{v_2,v_k\}$, and $V(P_{k-1})=V(C_{k-1})=\{v_2,\dots,v_k\}$. Note that $H_v\setminus v=C_{k-1}\cup_{e'} K_{m}$. Therefore by induction, $\depth(S_H/J_{H_v})=n$, $\depth(S_H'/J_{H_v\setminus v})=n-1$ and $\beta_{n,n+k-2}(S_H/J_{H_v})$, $\beta_{n-1,n+k-3}(S_H'/J_{H_v\setminus v})$ are the unique extremal Betti numbers. Thus it follows from \eqref{betti-product} that $\beta_{n+1,n+k-1}(S_H/((x_v,y_v)+J_{H_v\setminus v}))$ is the unique extremal Betti number. Since $\iv(H\setminus v)=k-2$, it follows from \cite[Theorem 1.1]{her1} that $\depth(S_H'/J_{H\setminus v})=n$, and hence by \cite[Theorem 8]{her2} and \eqref{betti-product}, $\beta_{n,n+k-1}(S_H/((x_v,y_v)+J_{H\setminus v}))$ is the unique extremal Betti number. As $v$ is not a simplicial vertex, we apply Lemma \ref{depth-lemma} on the short exact sequence \eqref{ohtani-ses} for the pair $(H,v)$ and get that $\depth(S_H/J_H)\geq n$. 
	Considering the long exact sequence of Tor \eqref{ohtani-tor} for $i=n$ and in graded degree $j=k-1$, we get
	\[ 0\longrightarrow \Tor_{n+1}^{S_H}\left(\frac{S_H}{((x_v,y_v)+J_{H_v\setminus v})},\K\right)_{n+k-1}\longrightarrow \Tor_{n}^{S_H}\left(\frac{S_H}{J_H},\K \right)_{n+k-1}\longrightarrow \cdots
	\]
	which implies that $\beta_{n,n+k-1}(S_H/J_H)\neq 0$. Therefore by Auslander-Buchsbaum formula, $\depth(S_H/J_H)\leq n$. Hence, $\depth(S_H/J_H)=n$ and $\beta_{n,n+k-1}(S_H/J_H)$ is the unique extremal Betti number of $S_H/J_H$ as $\reg(S_H/J_H)=k-1$, by \cite[Proposition 3.11]{JAR2}.
	\end{proof}
	Let $M$ be a graded $S$-module. Then the \textit{Betti polynomial} of M is defined as $\sum_{i,j} \beta_{i,j}(M)s^it^j$ and denoted by $B_M(s,t)$.
	
	A graph $G$ is said to be a \textit{decomposable graph} if $G$ is the clique sum of subgraphs $G_1$ and $G_2$ along a simplicial vertex i.e., $G=G_1\cup_v G_2$, where $v$ is a simplicial vertex of $G_1$ and $G_2$. If $G$ is not decomposable, then it is called an \textit{indecomposable graph}. We now recall the following result due to Herzog and Rinaldo.
	\begin{proposition}\cite[Proposition 3]{her2}\label{betti-decomposition}
		Let $G=G_1\cup G_2$ be a decomposable graph. Then 
		\[ B_{S/J_G}(s,t)=B_{S_{G_1}/J_{G_1}}(s,t) B_{S_{G_2}/J_{G_2}}(s,t). \]
	\end{proposition}
	As a corollary of the above Proposition, we get that if $G=G_1\cup \cdots \cup G_l$ is a decomposition into indecomposable graphs $G_i$, then $\pd(S/J_G)=\sum_{i=1}^{l} \pd(S_{G_i}/J_{G_i})$ and $\reg(S/J_G)=\sum_{i=1}^{l} \reg(S_{G_i}/J_{G_i})$. These two equalities follow from \cite[Theorem 2.7]{Rinaldo-Rauf} and \cite[Theorem 3.1]{JNR1}, respectively, as well. Also, if for each $i=1,\dots,l$, $\beta_{p_i,p_i+r_i}(S_{G_i}/J_{G_i})$ is an extremal Betti number of $S_{G_i}/J_{G_i}$, then $\beta_{p,p+r}(S/J_G)=\prod_{i=1}^{l}\beta_{p_i,p_i+r_i}(S_{G_i}/J_{G_i})$ is an extremal Betti number of $S/J_G$, where $p=\sum_{i=1}^{l}p_i$ and $r=\sum_{i=1}^{l}r_i$. Therefore to find the regularity, projective dimension, and extremal Betti number of $G$, it is enough to consider $G$ to be an indecomposable graph. So for the rest of the section, we assume that $G$ is an indecomposable graph.
	
	For a connected graph $G$, $\kappa(G)\geq 1$, and so by \cite[Theorems 3.19 and 3.20]{Arindam}, $\depth(S/J_G)\leq n+1$. Let $H=C_k\cup_e K_m$ be the clique sum of $C_k$ and $K_m$ along an edge $e$  for $k\geq 3$, $m\geq 2$. Let $G$ be the clique sum of $H$ and a forest along some vertices of $H$. We now study the depth of $S/J_G$ and prove that $n\leq \depth(S/J_G)$. Therefore, $\depth(S/J_G)\in \{n,n+1\}$. We then characterize $G$ with $\depth(S/J_G)=n$ and $\depth(S/J_G)=n+1$. Also, we obtain a lower bound for the regularity and show that if there are trees attached to each vertex of $C_k$, then $\iv(G)+1\leq \reg (S/J_G)$, otherwise $\iv(G)-1\leq \reg (S/J_G)$ by computing its one distinguished extremal Betti number. 
\begin{theorem}\label{depth-girth}
	Let $k\geq 3$ and $m\geq 2$. Let $G$ be the clique sum of $H=C_k\cup_e K_m$ and a forest along some vertices of $H$. Then $\depth(S/J_G)\geq n$. Let $A=\{u\in V(C_k): \text{there is a tree incident on }u\}$. If $e\cap A \neq \emptyset$ and $G[A]$ is connected with $k-2\leq |A|$, then $\depth(S/J_G)=n+1$.
\end{theorem}
\begin{proof}
	 Let $e=\{v,v_2\}$. By Lemma \ref{ohtani-lemma}, $J_G=J_{G_v}\cap ((x_v,y_v)+J_{G\setminus v})$, where $G\setminus v$  is a block graph on $n-1$ vertices. So by \cite[Theorem 1.1]{her1}, $\depth(S'/J_{G\setminus v})\geq n$. We prove the theorem by induction on $k$. For the case $k=3$, it can be noted that $G_v$ and $G_v\setminus v$ are both block graphs on $n$ and $n-1$ vertices, respectively. Thus it follows from \cite[Theorem 1.1]{her1} that $\depth(S/J_{G_v})=n+1$ and $\depth(S'/J_{G_v\setminus v})=n$.
	Consider the short exact sequence \eqref{ohtani-ses} and apply Lemma \ref{depth-lemma} to get that $\depth(S/J_G)\geq n$. 
	Now suppose that there is one tree attached to $v$. Then $G\setminus v$ is a disconnected block graph, and hence by \cite[Theorem 1.1]{her1}, $\depth(S'/J_{G\setminus v})\geq n+1$. Therefore, by Lemma \ref{depth-lemma} and the short exact sequence \eqref{ohtani-ses}, we have $\depth(S/J_G)=n+1$.
	 
	We now assume that $k\geq 4$. Let $K'$ and $K''$ be complete graphs on vertex sets $N_G[v]$ and $N_G(v)$, respectively. Also, let $H'=C_{k-1}\cup_{e'}K'$ and $H''=C_{k-1}\cup_{e'}K''$, where $e'=\{v_2,v_k\}$ and $V(C_{k-1})=\{v_2,\dots,v_k\}$. 
	Then it can be observed that $G_v$ is the clique sum of $H'$ and a forest along some vertices of $G$. Also, $G_v\setminus v$ is the clique sum of $H''$ and a forest along some vertices of $G$. Therefore, by induction $\depth(S/J_{G_v})\geq n$ and $\depth(S'/J_{G_v\setminus v})\geq n-1$. Thus, by applying Lemma \ref{depth-lemma} on the short exact sequence \eqref{ohtani-ses}, we get $\depth(S/J_G)\geq n$.
	Suppose now that $v\in A$ and $G[A]$ is connected with $k-2\leq |A|$. Then $G_v[A\setminus \{v\}]$ and $G_v\setminus v[A\setminus \{v\}]$ are both connected with $k-3\leq |A\setminus \{v\}|$. Hence, by induction $\depth(S/J_{G_v})=n+1$ and $\depth(S'/J_{G_v\setminus v})=n$. By \cite[Theorem 1.1]{her1}, we have $\depth(S'/J_{G\setminus v})\geq n+1$. Therefore it follows from Lemma \ref{depth-lemma} and the short exact sequence \eqref{ohtani-ses} that $\depth(S/J_G)=n+1$.
\end{proof}

Let $H=C_3\cup_e K_m$ for $m\geq 2$. Let $G$ be the clique sum of $H$ and a forest along some vertices of $H$. If $m=2$, then $G$ is a block graph. Ene, Herzog and Hibi \cite[Theorem 1.1]{her1} proved that in this case $\depth(S/J_G)=n+1$. Let $m\geq 3$. In Theorem \ref{depth-girth} we proved that $\depth(S/J_G)\geq n$, and if there are trees attached to either one vertex of $e$ or both the vertices of $e$, then $\depth(S/J_G)=n+1$. If there are no trees attached to any vertex of $e$, then $G$ is a generalized block graph. Hence it follows from \cite[Theorem 3.2]{KM-CA} and \cite[Theorem 3.4]{Arv-GBG} that $\depth(S/J_G)=n$ and $\beta_{n,n+\iv(G)}(S/J_G)$ is an extremal Betti number. Now we consider the case when trees are attached to at least one of the vertices of $e$.
	\begin{theorem}\label{depth-girth3}
	    Let $H=C_3\cup_e K_m$ for $m\geq 3$. Let $G$ be the clique sum of $H$ and a forest along some vertices of $H$. If there are trees attached to one vertex of $e$, then $\beta_{n-1,n-1+\iv(G)}(S/J_G)$ is an extremal Betti number and if there are trees attached to both the vertices of $e$, then $\beta_{n-1,n-1+\iv(G)+1}(S/J_G)$ is an extremal Betti number. In particular, $\iv(G)\leq \reg(S/J_G)$.
	\end{theorem}
	\begin{proof}
		Let $e=\{v,v_2\}$ and suppose that there are trees attached to $v$ in $G$. Then $G\setminus v$ is a disconnected block graph on $n-1$ vertices. By virtue of \cite[Theorem 1.1]{her1} and \eqref{betti-product}, we have $p=\pd(S/((x_v,y_v)+J_{G\setminus v}))\leq n-1$. By Lemma \ref{ohtani-lemma}, $J_G=J_{G_v}\cap ((x_v,y_v)+J_{G\setminus v}).$ Note that $G_v$ and $G_v\setminus v$ are block graphs on $n$ and $n-1$ vertices respectively. Therefore it follows from \cite[Theorem 6]{her2} and \eqref{betti-product} that $\beta_{p,p+\iv(G\setminus v)+1}(S/((x_v,y_v)+J_{G\setminus v}))$,  $\beta_{n-1,n-1+\iv(G_v)+1}(S/J_{G_v})$ and $\beta_{n,n+\iv(G_v\setminus v)+1}(S/((x_v,y_v)+J_{G_v\setminus v}))$ are extremal Betti numbers.
		We consider the long exact sequence \eqref{ohtani-tor} for $i=n-1$
		\begin{multline}\label{ohtani-tor1}
		0 \longrightarrow \Tor_{n}^{S}\left(\frac{S}{(x_v,y_v)+J_{G_v\setminus v}},\K\right)_{n-1+j} \longrightarrow \Tor_{n-1}^{S}\left( \frac{S}{J_G},\K\right)_{n-1+j}\longrightarrow \\ \longrightarrow \Tor_{n-1}^{S}\left( \frac{S}{(x_v,y_v)+J_{G\setminus v}},\K\right)_{n-1+j} 
		\oplus \Tor_{n-1}^{S}\left(\frac{S}{J_{G_v}},\K\right)_{n-1+j} \longrightarrow \cdots
		\end{multline}
		It is known \cite[Lemma 3.2]{Arv-Jaco} that $\iv(G)> \iv(G_v)=\iv(G_v\setminus v)$ and $\iv(G)> \iv(G\setminus v)$. Thus $\beta_{n-1,n-1+j}(S/J_{G_v})=0=\beta_{n-1,n-1+j}(S/((x_v,y_v)+J_{G\setminus v}))$ for $j\geq \iv(G)+1$. Therefore, we obtain
		\begin{align*}
		\Tor_{n}^{S}\left(\frac{S}{(x_v,y_v)+J_{G_v\setminus v}},\K\right)_{n-1+j} \simeq \Tor_{n-1}^{S}\left( \frac{S}{J_G},\K\right)_{n-1+j} \text{ for } j\geq \iv(G)+1.
		\end{align*}
		 If there is  no tree attached to $v_2$, then $\iv(G)=\iv(G_v\setminus v)+2$. Therefore it follows from the equation \eqref{ohtani-tor1} that $\beta_{n-1,n-1+\iv(G)}(S/J_G)\neq 0$ and $\beta_{n-1,n-1+j}(S/J_G)=0$ for $j\geq \iv(G)+1$. If there is a tree attached to $v_2$, then $\iv(G)=\iv(G_v\setminus v)+1$, and similarly, we have $\beta_{n-1,n-1+\iv(G)+1}(S/J_G)\neq 0$ and $\beta_{n-1,n-1+j}(S/J_G)=0$ for $j\geq \iv(G)+2$. By Theorem \ref{depth-girth}, $\pd(S/J_G)=n-1$, and hence, either $\beta_{n-1,n-1+\iv(G)}(S/J_G)$ or $\beta_{n-1,n-1+\iv(G)+1}(S/J_G)$ is an extremal Betti number of $S/J_G$, as desired.
		\end{proof}
		\noindent
		For $k=3$, we proved that $\depth(S/J_G)=n+1$ if and only if there are trees attached to at least one vertex of $e$ and in this case either $\beta_{n-1,n-1+\iv(G)}(S/J_G)$ or $\beta_{n-1,n-1+\iv(G)+1}(S/J_G)$ is an extremal Betti number. From now on, we assume that $k\geq 4$. Let $H=C_k\cup_e K_m$ for $m\geq 2$. Let $G$ be the clique sum of $H$ and a forest along some vertices $H$. First, we compute one distinguished extremal Betti number for the class of graphs $G$ with $\depth(S/J_G)=n+1$, considered in Theorem \ref{depth-girth}.

		\begin{theorem}\label{depthn+1-girth}
			Let $H=C_k\cup_e K_m$ for $k\geq 4$ and $m\geq 2$. Also, let $G$ be the clique sum of $H$ and a forest along some vertices of $H$. Let $A=\{u\in V(C_k): \text{there is a tree incident on }u\}$. If $e\cap A \neq \emptyset$ and $G[A]$ is connected with $k-2\leq |A|\leq k-1$, then either $\beta_{n-1,n-1+\iv(G)-1}(S/J_G)$ or $\beta_{n-1,n-1+\iv(G)}(S/J_G)$ is an extremal Betti number. If $A=V(C_k)$, then $\beta_{n-1,n-1+\iv(G)+1}(S/J_G)$ is an extremal Betti number. In particular, $\iv(G)-1\leq \reg(S/J_G)$.
		\end{theorem}
		\begin{proof}
			Let $e=\{v,v_2\}$. Suppose that $A\cap e\neq \emptyset$, $G[A]$ is connected, and $k-2\leq |A|\leq k-1$. Since $A\cap e\neq \emptyset$, we may assume that $v\in A$. Then $G\setminus v$ is a disconnected block graph on $n-1$ vertices, and hence it follows from \cite[Theorem 6]{her2} and \eqref{betti-product} that $\beta_{p,p+\iv(G\setminus v)+1}(S/((x_v,y_v)+J_{G\setminus v}))$ is an extremal Betti number of $S/((x_v,y_v)+J_{G\setminus v})$, where $p=\pd(S/((x_v,y_v)+J_{G\setminus v}))$. We prove the assertion by induction on $k$. First assume that $k=4$. By Lemma \ref{ohtani-lemma}, $J_G=J_{G_v}\cap ((x_v,y_v)+J_{G\setminus v})$ where $G_v$ belong to the class of graphs considered in Theorem \ref{depth-girth3}. Therefore, $G_v\setminus v$ also belong to the class of graphs considered in Theorem \ref{depth-girth3}. Hence, either $\beta_{n-1,n-1+\iv(G_v)}(S/J_{G_v})$ or $\beta_{n-1,n-1+\iv(G_v)+1}(S/J_{G_v})$ is an extremal Betti number of $S/J_{G_v}$ and either $\beta_{n,n+\iv(G_v\setminus v)}(S/((x_v,y_v)+J_{G_v\setminus v}))$ or $\beta_{n,n+\iv(G_v\setminus v)+1}(S/((x_v,y_v)+J_{G_v\setminus v}))$ is an extremal Betti number of $S/((x_v,y_v)+J_{G_v\setminus v})$.
			It is known \cite[Lemma 3.2]{Arv-Jaco} that $\iv(G)> \iv(G_v)=\iv(G_v\setminus v)$ and $\iv(G)> \iv(G\setminus v)$. By virtue of \cite[Theorem 1.1]{her1}, $p\leq n-1$. Hence, we have $\beta_{n-1,n-1+j}(S/J_{G_v})=0=\beta_{n-1,n-1+j}(S/((x_v,y_v)+J_{G\setminus v}))$ for $j\geq \iv(G)+1$. Therefore, it follows from the equation \eqref{ohtani-tor1} that for $j\geq \iv(G)+1$,
			\begin{align}\label{ohtani-tor2}
			\Tor_{n}^{S}\left(\frac{S}{(x_v,y_v)+J_{G_v\setminus v}},\K\right)_{n-1+j} \simeq \Tor_{n-1}^{S}\left( \frac{S}{J_G},\K\right)_{n-1+j}.
			\end{align}
			\textbf{Case 1:} Let $A=\{v,v_2\}$ or $A=\{v,v_4\}$. By Theorem \ref{depth-girth3} and the equation \eqref{betti-product}, we get that $\beta_{n,n+\iv(G_v\setminus v)}(S/((x_v,y_v)+J_{G_v\setminus v}))$ is an extremal Betti number. In this case, $\iv(G)=\iv(G_v\setminus v)+2$. Therefore, it follows from \eqref{ohtani-tor1} that $\beta_{n-1,n-1+\iv(G)-1}(S/J_G)\neq 0$.
			\vskip 1mm
			\noindent
			\textbf{Case 2:} If $A=\{v,v_2,v_3\}$ or $A=\{v,v_4,v_3\}$, then by Theorem \ref{depth-girth3} and \eqref{betti-product}, we get that $\beta_{n,n+\iv(G_v\setminus v)}(S/((x_v,y_v)+J_{G_v\setminus v}))$ is an extremal Betti number. In this case, $\iv(G)=\iv(G_v\setminus v)+1$. Therefore, by putting $j=\iv(G)$ in \eqref{ohtani-tor1}, we have $\beta_{n-1,n-1+\iv(G)}(S/J_G)\neq 0$.
			\vskip 1mm
			\noindent
			\textbf{Case 3:} If $A=\{v,v_2,v_4\}$, then by Theorem \ref{depth-girth3} and \eqref{betti-product}, $\beta_{n,n+\iv(G_v\setminus v)+1}(S/((x_v,y_v)+J_{G_v\setminus v}))$ is an extremal Betti number. In this case, $\iv(G)=\iv(G_v\setminus v)+2$. Thus it follows from the equation \eqref{ohtani-tor1} that $\beta_{n-1,n-1+\iv(G)}(S/J_G)\neq 0$.
			
			For all the above three cases, it follows from \eqref{ohtani-tor2} that $\beta_{n-1,n-1+j}(S/J_G)=0$ for $j\geq \iv(G)+1$. Hence, either $\beta_{n-1,n-1+\iv(G)-1}(S/J_G)$ or $\beta_{n-1,n-1+\iv(G)}(S/J_G)$ is an extremal Betti number of $S/J_G$.
			\vskip 1mm
			\noindent
			\textbf{Case 4:} If $A=V(C_4)$, then by Theorem \ref{depth-girth3} and the fact $\iv(G)=\iv(G_v\setminus v)+1$, we have $\beta_{n,n+\iv(G)}(S/((x_v,y_v)+J_{G_v\setminus v}))$ is an extremal Betti number of $S/((x_v,y_v)+J_{G_v\setminus v})$. Hence, it follows from \eqref{ohtani-tor2} that $\beta_{n-1,n+\iv(G)}(S/J_G)$ is an extremal Betti number of $S/J_G$.
			
			Now assume that $k\geq 5$. Let $K'$ and $K''$ denote complete graphs on vertex sets $N_G[v]$ and $N_G(v)$ respectively. Also, let $H'=C_{k-1}\cup_{e'}K'$ and $H''=C_{k-1}\cup_{e'}K''$, where $e'=\{v_2,v_k\}$ and $V(C_{k-1})=\{v_2,\dots,v_k\}$. 
			Then $G_v$ (resp. $G_v\setminus v$) is the clique sum of $H'$ (resp. $H''$) and a forest along some vertices of $G$. Clearly,  $G_v[A\setminus \{v\}]$ and $G_v\setminus v[A\setminus \{v\}]$ are both connected with $k-3\leq |A\setminus \{v\}|\leq k-2$, and so $G_v$ and $G_v\setminus v$ satisfy induction hypotheses. Therefore by induction either $\beta_{n-1,n-1+\iv(G_v)-1}(S/J_{G_v})$ or $\beta_{n-1,n-1+\iv(G_v)}(S/J_{G_v})$ is an extremal Betti number. Also, by induction and \eqref{betti-product}, either $\beta_{n,n+\iv(G_v\setminus v)-1}(S/((x_v,y_v)+J_{G_v\setminus v}))$ or $\beta_{n,n+\iv(G_v\setminus v)}(S/((x_v,y_v)+J_{G_v\setminus v}))$ is an extremal Betti number. Note that $\iv(G)=\iv(G_v)+1=\iv(G_v\setminus v)+1$. Therefore the assertion follows from the equations \eqref{ohtani-tor1} and \eqref{ohtani-tor2}. Suppose now that $A=V(C_k)$. Then clearly trees are attached to all the vertices of $C_{k-1}$ in both $G_v$ and $G_v\setminus v$. Thus by induction and \eqref{betti-product}, $\beta_{n-1,n-1+\iv(G_v)+1}(S/J_{G_v})$ and $\beta_{n,n+\iv(G_v\setminus v)+1}(S/((x_v,y_v)+J_{G_v\setminus v}))$ are extremal Betti numbers of $S/J_{G_v}$ and $S/((x_v,y_v)+J_{G_v\setminus v})$ respectively. Therefore it follows from \eqref{ohtani-tor2} and the fact $\iv(G)=\iv(G_v\setminus v)+1$ that $\beta_{n-1,n+\iv(G)}(S/J_G)$ is an extremal Betti number of $S/J_G$.
		\end{proof}
		
		By Theorem \ref{depth-girth}, we have that $\depth(S/J_G)\geq n$. We now characterize graphs attaining the lower bound. Also, we give a lower bound for the regularity of $S/J_G$ by computing one distinguished extremal Betti number of $S/J_G$. First, we consider the case $m\geq 3$.

		\begin{theorem}\label{depthn-girth}
       Let $H=C_k\cup_e K_m$ for $k\geq 4$ and $m\geq 3$. Let $G$ be the clique sum of $H$ and a forest along some vertices of $H$. Let $A=\{u\in V(C_k): \text{there is a tree incident on }u\}$. Suppose either $A \cap e = \emptyset$ or if $A \cap e \neq \emptyset$, then $A$ does not contain any $k-2$ consecutive vertices. Then either $\beta_{n,n+\iv(G)-1}(S/J_G)$ or $\beta_{n,n+\iv(G)}(S/J_G)$ is an extremal Betti number of $S/J_G$. In particular, $\depth(S/J_G)=n$ and $\iv(G)-1\leq \reg(S/J_G)$.
		\end{theorem}
		\begin{proof}
			Let $e=\{v,v_2\}$. Then $N_{C_k}(v)=\{v_2,v_k\}$. We proceed by induction on $k$. Let $k=4$. First assume that $A \cap e \neq \emptyset$ and $A$ does not contain any $2$ consecutive vertices. If $v \in A$, then $v$ is an internal vertex in $G$ and by Lemma \ref{ohtani-lemma}, we can write $J_G=J_{G_v}\cap ((x_v,y_v)+J_{G\setminus v})$. It can be noted that $G_v$ and $G_v\setminus v$ are generalized block graphs on $n$ and $n-1$ vertices respectively. If $v_2 \in A$, then $v_2$ is an internal vertex in $G$ and it follows from Lemma \ref{ohtani-lemma} that $J_G=J_{G_{v_2}}\cap ((x_{v_2},y_{v_2})+J_{G\setminus {v_2}})$. We can make similar conclusion about $G_{v_2}$ and $G_{v_2}\setminus v_2$. Now, suppose $A \cap e = \emptyset$. If $A =\{v_3\}$, then $G_v$ and $G_v\setminus v$ are generalized block graphs and if $A =\{v_4\}$, then $G_{v_2}$ and $G_{v_2}\setminus v_2$ are generalized block graphs. If $A\cap e\neq \emptyset$ and $v\in A$ or $A=\{v_3\}$, then set $w=v$. If $A\cap e\neq \emptyset$ and $v_2\in A$ or $A=\{v_4\}$, then set $w=v_2$.
			Then by \cite[Theorem 3.2]{KM-CA} and \eqref{betti-product}, $\pd(S/J_{G_w})=n$ and $\pd(S/((x_w,y_w)+J_{G_w\setminus w}))=n+1$. Hence it follows from \cite[Theorem 3.4]{Arv-GBG} that $\beta_{n,n+\iv(G_w)}(S/J_{G_w})$ and $\beta_{n+1,n+1+\iv(G_w\setminus w)}(S/((x_w,y_w)+J_{G_w\setminus w}))$ are extremal Betti numbers of $S/J_{G_w}$ and $S/((x_w,y_w)+J_{G_w\setminus w})$ respectively. Since $w$ is not a simplicial vertex, we consider the long exact sequence \eqref{ohtani-tor} for $i=n$: 
			\begin{multline}\label{ohtani-torn1}
				0 \longrightarrow \Tor_{n+1}^{S}\left(\frac{S}{(x_w,y_w)+J_{G_w\setminus w}},\K\right)_{n+j} \longrightarrow \Tor_{n}^{S}\left( \frac{S}{J_G},\K\right)_{n+j}\longrightarrow \\ \longrightarrow \Tor_{n}^{S}\left( \frac{S}{(x_w,y_w)+J_{G\setminus w}},\K\right)_{n+j} 
				\oplus \Tor_{n}^{S}\left(\frac{S}{J_{G_w}},\K\right)_{n+j} \longrightarrow \cdots
			\end{multline}
			By virtue of \cite[Lemma 3.2]{Arv-Jaco}, we have $\iv(G)> \iv(G_w)=\iv(G_w\setminus w)$ and $\iv(G)> \iv(G\setminus w)$. Hence, $\beta_{n,n+j}(S/J_{G_w})=0$ for $j\geq \iv(G)$. Since $G\setminus w$ is a block graph on $n-1$, by \cite[Theorem 1.1]{her1}, $\pd(S/((x_w,y_w)+J_{G\setminus w}))\leq n$. Now it follows from \cite[Theorem 6]{her2} and \eqref{betti-product} that $\beta_{n,n+j}(S/((x_w,y_w)+J_{G\setminus w}))=0$ for $j\geq \iv(G)+1$. Therefore, we get the isomorphism:
			\begin{align}\label{ohtani-torn2}
				\Tor_{n+1}^{S}\left(\frac{S}{(x_w,y_w)+J_{G_w\setminus w}},\K\right)_{n+j} \simeq \Tor_{n}^{S}\left( \frac{S}{J_G},\K\right)_{n+j} \text{ for} j\geq \iv(G)+1.
			\end{align}
			If $v_3\in A$ or $v_4\in A$, then note that $\iv(G)=\iv(G_w\setminus w)+1$, otherwise $\iv(G)=\iv(G_w\setminus w)+2$. Therefore it follows from \eqref{ohtani-torn1} and \eqref{ohtani-torn2} that either $\beta_{n,n+\iv(G)}(S/J_G)\neq 0$ or $\beta_{n,n+\iv(G)-1}(S/J_G)\neq 0$ and $\beta_{n,n+j}(S/J_G)=0$ for $j\geq \iv(G)+1$. Hence, either $\beta_{n,n+\iv(G)-1}(S/J_G)$ or $\beta_{n,n+\iv(G)}(S/J_G)$ is an extremal Betti number of $S/J_G$.
			
			Now the last case is $A =\{v_3, v_4\}$. Then $G_v$ and $G_v\setminus v$ belong to class of graphs considered in Theorem \ref{depth-girth3}. Hence, $\beta_{n-1,n-1+\iv(G_v)}(S/J_{G_v})$ and $\beta_{n,n+\iv(G_v\setminus v)}(S/((x_v,y_v)+J_{G_v\setminus v}))$ are extremal Betti numbers. Note that $\iv(G)=\iv(G_v\setminus v)+1$. Therefore, $\beta_{n,n+j}(S/((x_v,y_v)+J_{G_v\setminus v}))=0$ for $j\geq \iv(G)$. Since $G\setminus v$ is a block graph on $n-1$, by \cite[Theorem 1.1]{her1}, $\pd(S/((x_v,y_v)+J_{G\setminus v}))=n$.  Therefore, it follows from the long exact sequence \eqref{ohtani-tor} that 
			\begin{align*}
		    \Tor_{n}^{S}\left( \frac{S}{J_G},\K\right)_{n+j}\simeq \Tor_{n}^{S}\left( \frac{S}{(x_v,y_v)+J_{G\setminus v}},\K\right)_{n+j} \text{ for } j\geq \iv(G).
			\end{align*}
		    Since $\iv(G)=\iv(G\setminus v)+1$, by the help of \cite[Theorem 6]{her2} and the equation \eqref{betti-product}, we get that $\beta_{n,n+\iv(G)}(S/((x_v,y_v)+J_{G\setminus v}))$ is an extremal Betti number of $S/((x_v,y_v)+J_{G\setminus v})$. Therefore, $\beta_{n,n+\iv(G)}(S/J_G)$ is an extremal Betti number of $S/J_G$.
		    
			Now we assume that $k\geq 5$. Suppose, either $A \cap e = \emptyset$ or if $A \cap e \neq \emptyset$, then $A$ does not contain any $k-2$ consecutive vertices. Let $K'$ and $K''$ be complete graphs on vertex sets $N_G[v]$ and $N_G(v)$ respectively. Also, let $H'=C_{k-1}\cup_{e'}K'$ and $H''=C_{k-1}\cup_{e'}K''$, where $e'=\{v_2,v_k\}$ and $V(C_{k-1})=\{v_2,\dots,v_k\}$. Then $G_v$ (resp. $G_v\setminus v$) is the clique sum of $H'$ (resp. $H''$) and a forest along some vertices of $G$. Obviously, $A\setminus \{v\}\subseteq V(C_{k-1})$ is the set of vertices at which trees are attached in both $G_v$ and  $G_v\setminus v$.
			\vskip 1mm
			\noindent
			\textbf{Case 1:} If $|A|\leq k-4$, then clearly $G_v$ and $G_v\setminus v$ satisfy induction hypotheses.
			\vskip 1mm
			\noindent
			\textbf{Case 2:} Let $|A|=k-3$. If $v\in A$, then also $G_v$ and $G_v\setminus v$ satisfy induction hypotheses. If $v\notin A$ and $v_2\in A$, then $G_{v_2}$ and $G_{v_2}\setminus v_2$ satisfy induction hypotheses. Let $v,v_2\notin A$. Then $v_3\in A$ or $v_k\in A$. If $v_3\in A$, then $G_v$, $G_v\setminus v$ and if $v_k\in A$, then $G_{v_2}$, $G_{v_2}\setminus v_2$ satisfy induction hypotheses.
			\vskip 1mm
			\noindent
			\textbf{Case 3:} Let $|A|=k-2$ with $A\cap e\neq \emptyset$. If $v\in A$, then $G_v$ and $G_v\setminus v$ satisfy induction hypotheses, and if $v\notin A$, then $G_{v_2}$ and $G_{v_2}\setminus v_2$ satisfy induction hypotheses. 
			
			If $G_v$ satisfies induction hypotheses, then set $w=v$, and if $G_{v_2}$ satisfies induction hypotheses then set $w=v_2$. Now we apply induction on $G_w$ and $G_w\setminus w$. Therefore, either $\beta_{n,n+\iv(G_w)-1}(S/J_{G_w})$ or $\beta_{n,n+\iv(G_w)}(S/J_{G_w})$ is an extremal Betti number of $S/J_{G_w}$. Also, by induction and the equation \eqref{betti-product}, either $\beta_{n+1,n+1+\iv(G_w\setminus w)-1}(S/((x_w,y_w)+J_{G_w\setminus w}))$ or $\beta_{n+1,n+1+\iv(G_w\setminus w)}(S/((x_w,y_w)+J_{G_w\setminus w}))$ is an extremal Betti number of $S/((x_w,y_w)+J_{G_w\setminus w})$. Here, it can be observed that $\iv(G)=\iv(G_w)+1=\iv(G_w\setminus w)+1$. Hence, $\beta_{n,n+j}(S/J_{G_w})=0$ for $j\geq \iv(G)$. Since $G\setminus w$ is a block graph, it follows from \cite[Theorem 1.1]{her1} and \cite[Theorem 6]{her2} that $\beta_{n,n+j}(S/((x_w,y_w)+J_{G\setminus w}))=0$ for $j\geq \iv(G)+1$. Therefore, we get from \eqref{ohtani-torn1} that:
			\begin{align}\label{ohtani-torn3}
			\Tor_{n+1}^{S}\left(\frac{S}{(x_w,y_w)+J_{G_w\setminus w}},\K\right)_{n+j} \simeq \Tor_{n}^{S}\left( \frac{S}{J_G},\K\right)_{n+j} \text{ for } j\geq \iv(G)+1.
			\end{align}
			Hence, it follows from \eqref{ohtani-torn1} and \eqref{ohtani-torn3} that either $\beta_{n,n+\iv(G)-1}(S/J_G)$ or $\beta_{n,n+\iv(G)}(S/J_G)$ is an extremal Betti number of $S/J_G$.
			\vskip 1mm
			\noindent
			\textbf{Case 4:} The last case is $A\cap e=\emptyset$ and $|A|=k-2$. Then $G_v$ and $G_v\setminus v$ are graphs such that there are trees attached to $k-2$ consecutive vertices with $v_k\in e'\cap A\setminus \{v\}$. Then by Theorem \ref{depthn+1-girth}, we have that either $\beta_{n-1,n-1+\iv(G_v)-1}(S/J_{G_v})$ or $\beta_{n-1,n-1+\iv(G_v)}(S/J_{G_v})$ is an extremal Betti number of $S/J_{G_v}$ and either $\beta_{n,n+\iv(G_v\setminus v)-1}(S/((x_v,y_v)+J_{G_v\setminus v}))$ or $\beta_{n,n+\iv(G_v\setminus v)}(S/((x_v,y_v)+J_{G_v\setminus v}))$ is an extremal Betti number of $S/((x_v,y_v)+J_{G_v\setminus v})$. Since, $\iv(G)>\iv(G_v\setminus v)$, $\beta_{n,n+j}(S/((x_v,y_v)+J_{G_v\setminus v}))=0$ for $j\geq \iv(G)$. By  \cite[Theorem 6]{her2} and \eqref{betti-product}, $\beta_{n,n+\iv(G\setminus v)+1}(S/((x_v,y_v)+J_{G\setminus v}))$ is an extremal Betti number as $G\setminus v$ is a block graph on $n-1$ vertices.
			Therefore, it follows from the long exact sequence \eqref{ohtani-tor} that: 
			\begin{align*}
			\Tor_{n}^{S}\left( \frac{S}{J_G},\K\right)_{n+j}\simeq \Tor_{n}^{S}\left( \frac{S}{(x_v,y_v)+J_{G\setminus v}},\K\right)_{n+j} \text{ for } j\geq \iv(G).
			\end{align*}
			Note that $\iv(G)=\iv(G\setminus v)+1$. Hence, $\beta_{n,n+\iv(G)}(S/J_G)$ is an extremal Betti number of $S/J_G$. Therefore, $\pd(S/J_G)\geq n$, and so we have $\depth(S/J_G)\leq n$. Hence, by Theorem \ref{depth-girth}, $\depth(S/J_G)=n$, as desired.
		\end{proof}
		\begin{remark}{\em
		Let $H=C_k\cup_e K_m$ for $k\geq 3$ and $m\geq 3$. Let $G$ be the clique sum of $H$ and a forest along some vertices of $H$. Let $A=\{u\in V(C_k): \text{there is a tree incident on u}\}$. By \cite[Theorems 3.19 and 3.20]{Arindam} and Theorem \ref{depth-girth}, $n\leq \depth(S/J_G)\leq n+1$. Moreover, it follows from Theorems \ref{depth-girth}, \ref{depth-girth3} and \ref{depthn-girth} that $A\cap e\neq \emptyset$ and $G[A]$ is connected with $k-2\leq |A|$ if and only if $\depth(S/J_G)=n+1$.
		}
		\end{remark}
		Now we consider the case $m=2$. Let $G$ be a unicyclic graph of girth $k(\geq 4)$. If there are trees attached to $k-2$ consecutive vertices, then by Theorem \ref{depth-girth}, $\depth(S/J_G)=n+1$. We now characterize unicyclic graph $G$ such that $\depth(S/J_G)=n$.
		\begin{theorem}\label{depthn-unicyclic}
			Let $G$ be a unicyclic graph of girth $k$ for $k\geq 4$. Let $A=\{u\in V(C_k): \text{there is a tree incident on }u\}$. If $A$ does not contain any $k-2$ consecutive vertices, then either $\beta_{n,n+\iv(G)-1}(S/J_G)$  or $\beta_{n,n+\iv(G)}(S/J_G)$ is an extremal Betti number of $S/J_G$. In particular, $\depth(S/J_G)=n$ and $\iv(G)-1\leq \reg(S/J_G)$.
		\end{theorem}
		\begin{proof}
		Suppose that $A$ does not contain any $k-2$ consecutive vertices. Since $A\neq \emptyset$, we may assume that $v\in A$. Then $G\setminus v$ is a disconnected block graph, and hence by \cite[Theorem 1.1]{her1} and \eqref{betti-product}, $p=\pd(S/((x_v,y_v)+J_{G\setminus v}))\leq n-1$. We prove the theorem by induction on $k$. First assume that $k=4$. Then $v_2,v_k\notin A$, and hence $G_v$ and $G_v\setminus v$ are generalized block graphs on $n$ and $n-1$ vertices respectively. By virtue of \cite[Theorem 3.2]{KM-CA} and \eqref{betti-product}, we get that $\pd(S/J_{G_v})=n$ and $\pd(S/((x_v,y_v)+J_{G_v\setminus v}))=n+1$. Hence, by \cite[Theorem 3.4]{Arv-GBG}, $\beta_{n,n+\iv(G_v)}(S/J_{G_v})$ is an extremal Betti number of $S/J_{G_v}$ and $\beta_{n+1,n+1+\iv(G_v\setminus v)}(S/((x_v,y_v)+J_{G_v\setminus v}))$ is an extremal Betti number of $S/((x_v,y_v)+J_{G_v\setminus v})$. Since $\iv(G)>\iv(G_v)$, $\beta_{n,n+j}(S/J_{G_v})=0$ for $j\geq \iv(G)$. Thus, we have the following isomorphism from the long exact sequence \eqref{ohtani-tor}:
		\begin{align}\label{ohtani-tor-unicyclic}
		\Tor_{n+1}^{S}\left(\frac{S}{(x_v,y_v)+J_{G_v\setminus v}},\K\right)_{n+j} \simeq \Tor_{n}^{S}\left( \frac{S}{J_G},\K\right)_{n+j} \text{ for} j\geq \iv(G).
		\end{align}
		If $v_3\in A$, then $\iv(G)=\iv(G_v\setminus v)+1$, otherwise $\iv(G)=\iv(G_v\setminus v)+2$. Therefore, it follows from \eqref{ohtani-tor} and \eqref{ohtani-tor-unicyclic} that either $\beta_{n,n+\iv(G)}(S/J_G)$ or $\beta_{n,n+\iv(G)-1}(S/J_G)$ is an extremal Betti number of $S/J_G$.
		 
		Now we assume that $k\geq 5$. Let $n_v=|N_G(v)|$. Then $n_v\geq 3$. Let $H'=C_{k-1}\cup_{e'}K_{n_v+1}$ and $H''=C_{k-1}\cup_{e'}K_{n_v}$ where $e'=\{v_2,v_k\}$ and $C_{k-1}$ is a cycle on $\{v_2,\dots,v_k\}$. Therefore, $G_v$ (resp. $G_v\setminus v$ ) is the clique sum of $H'$ (resp. $H''$ ) and a forest along some vertices of $G$. It can be easily seen that for $G_v$ ( resp. $G_v\setminus v$), $A\setminus v\subset V(C_{k-1})$ is the set of vertices along which trees are attached to $H'$ (resp. $H''$). Set $A'=A\setminus v$. If $v_2,v_k\notin A'$, then $A'\cap e'=\emptyset$. Otherwise, if $A'\cap e'\neq \emptyset$, then clearly $A'$ does not contain any $k-3$ consecutive vertices. Therefore, by Theorem \ref{depthn-girth}, either $\beta_{n,n+\iv(G_v)-1}(S/J_{G_v})$ or $\beta_{n,n+\iv(G_v)}(S/J_{G_v})$ is an extremal Betti number of $S/J_{G_v}$ and either $\beta_{n+1,n+1+\iv(G_v\setminus v)-1}(S/((x_v,y_v)+J_{G_v\setminus v}))$
		or $\beta_{n+1,n+1+\iv(G_v\setminus v)}(S/((x_v,y_v)+J_{G_v\setminus v}))$ is an extremal Betti number of $S/((x_v,y_v)+J_{G_v\setminus v})$. Note that $\iv(G)=\iv(G_v\setminus v)+1$. Now the assertion follows from \eqref{ohtani-tor} and \eqref{ohtani-tor-unicyclic}.
		\end{proof}
		Therefore, from Theorems \ref{depth-girth} and \ref{depthn-unicyclic}, we can conclude the following result for unicyclic graphs. 
		\begin{corollary}\label{depth-unicyclic}
			Let $G$ be a unicyclic graph on the vertex set $[n]$ of girth $k\geq 4$. Then $n\leq \depth(S/J_G)\leq n+1$. Moreover, if there are trees attached to $k-2$ consecutive vertices of the cycle in $G$, then $\depth(S/J_G)=n+1$, otherwise $\depth(S/J_G)=n$.
		\end{corollary}
		Also, we can conclude from Theorems \ref{depthn+1-girth} and \ref{depthn-unicyclic} that.
			\begin{corollary}\label{betti-unicyclic}
				Let $G$ be a unicyclic graph of girth $k\geq 4$ with $\pd(S/J_G)=p$. If trees are attached to every vertex of the cycle in $G$, then $\beta_{p,p+\iv(G)+1}(S/J_G)$ is an extremal Betti number of $S/J_G$, and hence $\iv(G)+1\leq \reg(S/J_G)$. Otherwise, either $\beta_{p,p+\iv(G)-1}(S/J_G)$ or $\beta_{p,p+\iv(G)}(S/J_G)$ is an extremal Betti number of $S/J_G$, and hence $\iv(G)-1\leq \reg(S/J_G)$.
			\end{corollary}
\section{Cycles with whiskers}\label{whisker}
Let $G$ be a graph with the vertex set $V(G)$ and $v\in V(G)$. Let $u_1,\dots,u_r$ be new vertices and we define $G\cup W^r(v)$ to be the graph with vertex set $V(G\cup W^r(v))=V(G)\cup \{u_1,\dots,u_r\}$ and edge set $E(G\cup W^r(v))=E(G)\cup \{\{u_i,v\}: 1\leq i\leq r\}$ i.e., $G\cup W^r(v)$ is the graph obtained from $G$ by attaching $r$ whiskers or pendant edges at the vertex $v$. If $r=0$, then $G\cup W^0(v)=G$.
Let $v_1,\dots,v_s\in V(G)$. Then in a similar way, we can attach $r_i$ whiskers at $v_i$ for $1\leq i\leq s$. We denote this graph by $G\cup (\cup_{i=1}^{s} W^{r_i}(v_i))$. Algebraic effect of attaching whiskers to a graph has already been studied for the case of monomial edge ideals, see \cite{Tai-Whiskers}, \cite{Fakhari-Whiskers} and \cite{Villarreal-Whiskers}. The algebraic effect of attaching whiskers to a graph has been studied also for binomial edge ideals. You can see, for instance, \cite{CDI16,Rinaldo-BMS,Rinaldo-Cactus,Zafar}. Here we study the binomial edge ideals of graphs with whiskers attached.

\vskip 2mm
\noindent
\begin{minipage}{\linewidth}
	\begin{minipage}{.4\linewidth}
		The graph $G$, given on the right, is $G=C_5\cup W^2(v_1)\cup W^2(v_2)\cup W^1(v_5)$.
	\end{minipage}
	\begin{minipage}{.6\linewidth}
		\captionsetup[figure]{labelformat=empty}
		\begin{figure}[H]
\begin{tikzpicture}[scale=1]
\draw (4.04,-2.14)-- (3.5,-2.66);
\draw (4.64,-2.68)-- (4.04,-2.14);
\draw (3.5,-2.66)-- (3.5,-3.43);
\draw (3.5,-3.43)-- (4.66,-3.43);
\draw (4.66,-3.43)-- (4.64,-2.68);
\draw (5.27,-2.58)-- (4.64,-2.68);
\draw (4.64,-2.68)-- (5.05,-2.14);
\draw (2.85,-2.56)-- (3.5,-2.66);
\draw (3.7,-1.7)-- (4.04,-2.14);
\draw (4.46,-1.62)-- (4.04,-2.14);
\begin{scriptsize}
\fill (4.04,-2.14) circle (1.5pt);
\fill (5.05,-2.14) circle (1.5pt);
\draw (4.1,-2.4) node {$v_1$};
\fill (3.5,-2.66) circle (1.5pt);
\draw (3.7,-2.75) node {$v_5$};
\fill (4.64,-2.68) circle (1.5pt);
\draw (4.4,-2.8) node {$v_2$};
\fill (3.5,-3.43) circle (1.5pt);
\draw (3.7,-3.27) node {$v_4$};
\fill (4.66,-3.43) circle (1.5pt);
\draw (4.45,-3.27) node {$v_3$};
\fill (2.85,-2.56) circle (1.5pt);
\fill (5.27,-2.58) circle (1.5pt);
\fill (3.7,-1.7) circle (1.5pt);
\fill (4.46,-1.62) circle (1.5pt);
\end{scriptsize}
\end{tikzpicture}
\caption{$G$}
\end{figure}
\end{minipage}
\end{minipage}

Let $G=C_k\cup (\cup_{i=1}^{k} W^{r_i}(v_i))$ for $r_i\geq 0$. Then $n=k+\sum_{i=1}^{k}r_i$. Let $A=\{v_i\in V(C_k): r_i\geq 1\}$. If $A=\emptyset$, then $G=C_k$ and in this case, Zafar and Zahid \cite[Corollary 16]{Zafar} proved that $\reg(S/J_G)=k-2$. So, in the rest of the section we assume that $A\neq \emptyset$. If $k=3$, then $1\leq \iv(G)\leq 3$, and hence by \cite[Theorem 8]{her2}, $2\leq \reg(S/J_G)=1+\iv(G)\leq 4$. In this section, we generalize their result for $k\geq 4$ and prove that the regularity of $S/J_G$ is bounded below by $k-1$ and bounded above by $k+1$. We then characterize graphs $G$ with $\reg(S/J_G)=k-1$, $\reg(S/J_G)=k$ and $\reg(S/J_G)=k+1$. We also classify $G$ which admits a unique extremal Betti number.
\begin{theorem}\label{reg-whisker-cliquesum}
	Let $H=K_m\cup_{e} C_k$ for $m\geq 2$, $k\geq 3$. Let $G=H\cup (\cup_{i=1}^{k} W^{r_i}(v_i))$ for $r_i\geq 0$, and suppose that $A=\{v_i\in V(C_k): r_i\geq 1\}$. If $A\neq \emptyset$, then $k-1\leq \reg(S/J_G)\leq k+1$.
\end{theorem}
\begin{proof}
	Note that $G$ contains an induced path of length $k-1$. Then the lower bound follows from \cite[Corollary 2.3]{MM}. Let $e=\{v,v_2\}$. Then $G\setminus v$ is the graph $K_{m-1}\cup_{v_2}P_{k-1}\cup
	(\cup_{i=2}^{k} W^{r_i}(v_i))$ with $r_1$ isolated vertices, where $V(P_{k-1})=\{v_2,\dots,v_k\}$. Since $\iv(G\setminus v)\leq k-1$, it follows from \cite[Theorem 8]{her2} and \eqref{betti-product} that $\reg(S/((x_v,y_v)+J_{G\setminus v}))\leq k$. We prove the upper bound by induction on $k$. Assume that $k=3$. If $m=2$, then the assertion follows from \cite[Theorem 8]{her2}. Suppose now that $m\geq 3$. Since $v$ is an internal vertex, by Lemma \ref{ohtani-lemma}, $J_G=J_{G_v}\cap ((x_v,y_v)+J_{G\setminus v})$, where $G_v=K_{m+r_1+1}\cup W^{r_2}(v_2)\cup W^{r_3}(v_3)$. Therefore, $G_v\setminus v=K_{m+r_1}\cup W^{r_2}(v_2)\cup W^{r_3}(v_3)$. Hence, by \cite[Theorem 8]{her2}, $\reg(S/J_{G_v})=\iv(G_v)+1=\iv(G_v\setminus v)+1=\reg(S/((x_v,y_v)+J_{G_v\setminus v}))\leq 3$. We apply Lemma \ref{regularity-lemma} on the short exact sequence \eqref{ohtani-ses} to get that $\reg(S/J_G)\leq 4$. If $A=\{v\}$, then $G_v=K_{m+r_1+1}$, $G_v\setminus v=K_{m+r_1}$ and $G\setminus v=K_{m-1}\cup _{v_2}P_{2}$ with $r_1$ isolated vertices. Therefore, $\reg(S/J_{G_v})=1=\reg(S/((x_v,y_v)+J_{G_v\setminus v}))$ and $\reg(S/((x_v,y_v)+J_{G\setminus v}))=2$. Thus it follows from Lemma \ref{regularity-lemma} and the short exact sequence \eqref{ohtani-ses} that $\reg(S/J_G)=2$.
	
	Now assume that $k\geq 4$. Note that $G_v=K_{m+r_1+1}\cup_{\{v_2,v_k\}}C_{k-1}\cup
	(\cup_{i=2}^{k} W^{r_i}(v_i))$ and $G_v\setminus v=K_{m+r_1}\cup_{\{v_2,v_k\}}C_{k-1}\cup
	(\cup_{i=2}^{k} W^{r_i}(v_i))$, where $V(C_{k-1})=\{v_2,\dots,v_k\}$. If $A=\{v\}$, then $G_v=K_{m+r_1+1}\cup_{\{v_2,v_k\}}C_{k-1}$ and $G_v\setminus v=K_{m+r_1}\cup_{\{v_2,v_k\}}C_{k-1}$ . Thus, by \cite[Theorem 3.12]{JAR2}, $\reg(S/J_{G_v})=\reg(S/((x_v,y_v)+J_{G_v\setminus v}))=k-2$. Also, in this case $G\setminus v=K_{m-1}\cup_{v_2}P_{k-1}$ with $r_1$ isolated vertices, and hence by Proposition \ref{betti-decomposition} and \eqref{betti-product}, $\reg(S/((x_v,y_v)+J_{G\setminus v}))\leq k-1$. Therefore, by Lemma \ref{regularity-lemma} and the short exact sequence \eqref{ohtani-ses}, we have that $\reg(S/J_G)\leq k-1$. Hence, if $A=\{v\}$, then $\reg(S/J_G)=k-1$. Let $v_i\in A$ for some $2\leq i\leq k$. Then by induction, $\reg(S/J_{G_v})\leq k$ and $\reg(S/((x_v,y_v)+J_{G_v\setminus v}))\leq k$. Thus, Lemma \ref{regularity-lemma} and the short exact sequence \eqref{ohtani-ses} together imply that $\reg(S/J_G)\leq k+1$.
	\end{proof}
	Considering $m=2$ in Theorem \ref{reg-whisker-cliquesum}, we obtain bounds for cycles with whiskers.
	\begin{corollary}\label{reg-whisker}
		Let $k\geq 3$ and $G=C_k\cup (\cup_{i=1}^{k} W^{r_i}(v_i))$, $r_i\geq 0$. Let $A=\{v_i\in V(C_k): r_i\geq 1\}$. If $A\neq \emptyset$, then $k-1\leq \reg(S/J_G)\leq k+1$. Moreover, if $|A|=1$, then $\reg(S/J_G)=k-1$.
	\end{corollary}
	We now characterize $G$ with $\reg(S/J_G)=k+1$. First, we prove the following Proposition.
	\begin{proposition}\label{reg-whiskerk-1_k}
		Let $H=K_m\cup_{e} C_k$ for $m\geq 2$, $k\geq 3$. Let $G=H\cup (\cup_{i=1}^{k} W^{r_i}(v_i))$ for $r_i\geq 0$, and suppose that $A=\{v_i\in V(C_k): r_i\geq 1\}$. If $e\nsubseteq A$, then $\reg(S/J_G)\leq k$.
	\end{proposition}
	\begin{proof}
		Let $e=\{v,v_2\}$. We assume that $v_2\notin A$, i.e., $r_2=0$. We proceed by induction on $k$. Note that $G\setminus v=K_{m-1}\cup_{v_2}P_{k-1}\cup (\cup_{i=3}^{k} W^{r_i}(v_i))$ with $r_1$ isolated vertices. Hence by \cite[Theorem 8]{her2} and \eqref{betti-product}, $\reg(S/((x_v,y_v)+J_{G\setminus v}))\leq k$. For $k=3$, $G_v=K_{m+r_1+1}\cup W^{r_3}(v_3)$ and $G_v\setminus v=K_{m+r_1}\cup W^{r_3}(v_3)$. Then by virtue of \cite[Theorem 8]{her2}, $\reg(S/J_{G_v})=\reg(S/((x_v,y_v)+J_{G_v\setminus v}))\leq 2$. Therefore, by applying Lemma \ref{regularity-lemma} on the short exact sequence \eqref{ohtani-ses}, we get $\reg(S/J_G)\leq 3$.
		Now suppose that $k\geq 4$. Then $G_v=K_{m+r_1+1}\cup_{\{v_2,v_k\}}C_{k-1}\cup(\cup_{i=3}^{k} W^{r_i}(v_i))$ and $G_v\setminus v=K_{m+r_1}\cup_{\{v_2,v_k\}}C_{k-1}\cup(\cup_{i=3}^{k} W^{r_i}(v_i))$. Hence by induction, $\reg(S/J_{G_v})\leq k-1$ and $\reg(S/((x_v,y_v)+J_{G_v\setminus v}))\leq k-1$. Therefore, it follows from Lemma \ref{regularity-lemma} and the short exact sequence \eqref{ohtani-ses} that $\reg(S/J_G)\leq k$.
	\end{proof}
	\begin{corollary}\label{reg-whiskerk+1}
		Let $k\geq 3$ and $G=C_k\cup (\cup_{i=1}^{k} W^{r_i}(v_i))$, $r_i\geq 0$. Let $A=\{v_i\in V(C_k): r_i\geq 1\}$. Then $A=V(C_k)$ if and only if $\reg(S/J_G)=k+1$. Moreover, in this case, $S/J_G$ admits a unique extremal Betti number.
	\end{corollary}
	\begin{proof}
		First, we assume that whiskers are attached at every vertex of $C_k$ i.e., $r_i\geq 1$ for all $1\leq i\leq k$. For $k=3$, by \cite[Theorem 8]{her2}, we have that $\beta_{n-1,n-1+4}(S/J_G)$ is the unique extremal Betti number. For $k\geq 4$, by Theorem \ref{depthn+1-girth}, $\beta_{n-1,n-1+k+1}(S/J_G)$ is an extremal Betti number, which further implies that $k+1\leq \reg(S/J_G)$. By Corollary \ref{reg-whisker}, $\reg(S/J_G)\leq k+1$. Hence, $\reg(S/J_G)=k+1$ and $S/J_G$ admits a unique extremal Betti number. For the converse part, suppose there exists $i\in [k]$ such that $r_i=0$. Then by Proposition \ref{reg-whiskerk-1_k}, $\reg(S/J_G)\leq k$, which is a contradiction. Hence, the assertion follows.
	\end{proof}
		If $\emptyset \neq A\subsetneq V(C_k)$, then by Corollary \ref{reg-whisker} and Proposition \ref{reg-whiskerk-1_k}, $k-1\leq \reg(S/J_G)\leq k$. We now characterize $G$ with $\reg(S/J_G)=k-1$ and $\reg(S/J_G)=k$.
		\begin{theorem}\label{reg-whiskerk-1}
			Let $k\geq 4$ and $G=C_k\cup (\cup_{i=1}^{k} W^{r_i}(v_i))$ for $r_i\geq 0$. Let $A=\{v_i\in V(C_k): r_i\geq 1\}$. If either $|A|=1$ or $|A|=2$ and vertices of $A$ are adjacent, then $\reg(S/J_G)=k-1$. Moreover, in this case, $S/J_G$ admits a unique extremal Betti number.
		\end{theorem}
		\begin{proof}
			Suppose $G=C_k\cup (\cup_{i=1}^{k} W^{r_i}(v_i))$ and whiskers are attached only at one vertex. Then by Corollary \ref{reg-whisker}, $\reg(S/J_G)=k-1$. Now assume that whiskers are attached only at two adjacent vertices of $C_k$, say $G=C_k\cup W^{r_1}(v)\cup W^{r_2}(v_2)$ for $r_1,r_2\geq 1$. Note that $G\setminus v$ is the graph $P_{k-1}\cup W^{r_2}(v_2)$ with $r_1$ isolated vertices, where $V(P_{k-1})=\{v_2,\dots,v_k\}$. Thus by \cite[Theorem 4.1]{CDI16} and \eqref{betti-product}, $\reg(S/((x_v,y_v)+J_{G\setminus v}))=k-1$. Here, $G_v=K_{r_1+3}\cup_{\{v_2,v_k\}}(C_{k-1}\cup W^{r_2}(v_2))$ and $G_v\setminus v=K_{r_1+2}\cup_{\{v_2,v_k\}}(C_{k-1}\cup W^{r_2}(v_2))$. Then it follows from the proof of Theorem \ref{reg-whisker-cliquesum} that $\reg(S/J_{G_v})=k-2$ and $\reg(S/((x_v,y_v)+J_{G_v\setminus v}))=k-2$. Therefore by Lemma \ref{regularity-lemma} and the short exact sequence \eqref{ohtani-ses}, $\reg(S/J_G)\leq k-1$. Hence, $\reg(S/J_G)=k-1$. Now it follows from Corollary \ref{betti-unicyclic} that $\beta_{p,p+k-1}(S/J_G)$ is the unique extremal Betti number of $S/J_G$, where $p=\pd(S/J_G)$.
		\end{proof}
		Now we prove that if $G$ does not belong to the class of graphs considered in Theorem \ref{reg-whiskerk-1}, then $\reg(S/J_G)=k$. To prove this, we first need to compute extremal Betti number of some intermediate graphs.
			\begin{proposition}\label{betti-cliquesum2}
				Let $G=C_k\cup_e K_m \cup_{e'}K_{m'}$ for $k,m,m'\geq 3$. Then $\beta_{n,n+k}(S/J_G)$ is the unique extremal Betti number of $S/J_G$.
			\end{proposition}
			\begin{proof}
				Let $e=\{v,v_2\}$. It is enough to prove that $n\leq \depth(S/J_G)$, $\reg(S/J_G)\leq k$ and $\beta_{n,n+k}(S/J_G)\neq 0$. We
				prove this by induction on $k$. Assume that $k=3$. Since $e\cap e'\neq \emptyset$, we assume that $e'=\{v,v_3\}$. Then, it can be seen that $G_v=K_n$, $G_v\setminus v=K_{n-1}$ and $G\setminus v=K_{m-1}\cup_{v_2}P_2\cup_{v_3}K_{m'-1}$. Thus, we have that $\depth((S/J_{G_v}))=n+1$, $\depth(S/((x_v,y_v)+J_{G_v\setminus v}))=n$ and $\reg((S/J_{G_v}))=\reg(S/((x_v,y_v)+J_{G_v\setminus v}))=1$. Moreover, $\beta_{n-1,n}(S/J_{G_v})$ and $\beta_{n,n+1}(S/((x_v,y_v)+J_{G_v\setminus v}))$ are the unique extremal Betti numbers of $S/J_{G_v}$ and $S/((x_v,y_v)+J_{G_v\setminus v})$, respectively. By Proposition \ref{betti-decomposition}, $\reg(S/((x_v,y_v)+J_{G\setminus v}))=3$. Also, it follows from \cite[Theorem 3.1]{her1} that $J_{G\setminus v}$ is Cohen-Macaulay. Hence, $\depth(S/((x_v,y_v)+J_{G\setminus v}))=n$ and $\beta_{n,n+3}(S/((x_v,y_v)+J_{G\setminus v}))$ is the unique extremal Betti number. Therefore, by applying Lemmas \ref{depth-lemma} and \ref{regularity-lemma} on the short exact sequence \eqref{ohtani-ses}, we get that $\depth(S/J_G)\geq n$ and $\reg(S/J_G)\leq 3$. 
				Now it follows from the long exact sequence \eqref{ohtani-tor} for $i=n$ and $j=3$ that 
				 \begin{align*} 
				 \Tor_{n}^{S}\left( \frac{S}{J_G},\K\right)_{n+3}\simeq \Tor_{n}^{S}\left(\frac{S}{(x_v,y_v)+J_{G\setminus v}},\K\right)_{n+3}\neq 0.
				 \end{align*}
				 Therefore, $\beta_{n,n+3}(S/J_G)\neq 0$.
				 Now assume that $k\geq 4$.  
				 \vskip 1mm
				 \noindent
				 \textbf{Case 1:} Let $e$ $\cap$ $e'\neq \emptyset$. Suppose $v\in e\cap e'$. Then $e'=\{v,v_k\}.$ Note that $G_v=K_{m+m'-1}\cup_{\{v_2,v_k\}}C_{k-1}$, $G_v\setminus v=K_{m+m'-2}\cup_{\{v_2,v_k\}}C_{k-1}$ and $G\setminus v=K_{m-1}\cup_{v_2}P_{k-1}\cup_{v_k}K_{m'-1}$. Thus, by virtue of \cite[Proposition 3.11]{JAR2}, we have $\reg(S/J_{G_v})=k-2=\reg(S/((x_v,y_v)+J_{G_v\setminus v}))$. By Proposition \ref{betti-cliquesum}, $\depth(S/J_{G_v})=n$, $\depth(S/((x_v,y_v)+J_{G_v\setminus v}))=n-1$ and $\beta_{n,n+k-2}(S/J_{G_v})$, $\beta_{n+1,n+1+k-2}(S/((x_v,y_v)+J_{G_v\setminus v}))$ are the unique extremal Betti numbers. It follows from Proposition \ref{betti-decomposition} that $\reg(S/((x_v,y_v)+J_{G\setminus v}))=k$. By virtue of \cite[Theorem 3.1]{her1}, $J_{G\setminus v}$ is Cohen-Macaulay. Therefore, $\depth(S/((x_v,y_v)+J_{G\setminus v}))=n$ and $\beta_{n,n+k}(S/((x_v,y_v)+J_{G\setminus v}))$ is the unique extremal Betti number. Hence, by using Lemmas \ref{depth-lemma} and \ref{regularity-lemma} on the short exact sequence \eqref{ohtani-ses}, we have $n\leq \depth(S/J_G)$ and $\reg(S/J_G)\leq k$. Consider the long exact sequence \eqref{ohtani-tor} for $i=n$, $j=k$ and we get that  
				 \begin{align*} 
				 \Tor_{n}^{S}\left( \frac{S}{J_G},\K\right)_{n+k}\simeq \Tor_{n}^{S}\left(\frac{S}{(x_v,y_v)+J_{G\setminus v}},\K\right)_{n+k}\neq 0.
				 \end{align*}
				 \vskip 1mm
				 \noindent
				 \textbf{Case 2:} Let $e\cap e'=\emptyset$. Let $e'=\{v_i,v_{i+1}\}$ for $i\geq 3$. It can be noted that $G_v=K_{m+1}\cup_{\{v_2,v_k\}}C_{k-1}\cup_{e'}K_{m'}$, $G_v\setminus v=K_{m}\cup_{\{v_2,v_k\}}C_{k-1}\cup_{e'}K_{m'}$ and $G\setminus v=K_{m-1}\cup_{v_2}P_{k-1}\cup_{e'}K_{m'}$. Thus, by induction and \eqref{betti-product}, $\beta_{n,n+k-1}(S/J_{G_v})$ and $\beta_{n+1,n+1+k-1}(S/((x_v,y_v)+J_{G_v\setminus v}))$ are the unique extremal Betti numbers of $S/J_{G_v}$ and $S/((x_v,y_v)+J_{G_v\setminus v})$ respectively. By virtue of Proposition \ref{betti-decomposition}, $\reg(S/((x_v,y_v)+J_{G\setminus v}))=k-1$. It is known \cite[Theorem 3.1]{her1} that $J_{G\setminus v}$ is Cohen-Macaulay, and hence $\depth(S/((x_v,y_v)+J_{G\setminus v}))=n$ and $\beta_{n,n+k-1}(S/((x_v,y_v)+J_{G\setminus v}))$ is the unique extremal Betti number. Therefore, by applying Lemmas \ref{depth-lemma}, \ref{regularity-lemma} on the short exact sequence \eqref{ohtani-ses}, we have $\depth(S/J_G)\geq n$ and $\reg(S/J_G)\leq k$. Also, it follows from the long exact sequence \eqref{ohtani-tor} for $i=n+1$ and in graded degree $j=k-1$ that
				 \begin{align*} \Tor_{n+1}^{S}\left(\frac{S}{(x_v,y_v)+J_{G_v\setminus v}},\K\right)_{n+1+k-1}\simeq \Tor_{n}^{S}\left( \frac{S}{J_G},\K\right)_{n+1+k-1}\neq 0. 
				 \end{align*}
				 Therefore, $\beta_{n,n+k}(S/J_G)\neq 0$, as required.
				\end{proof}
				\begin{proposition}\label{betti-cliquesum3}
					Let $m\geq 3$, $k\geq 4$ and $G=K_m\cup_e C_k \cup W^{r_1}(v)$ for $r_1\geq 1$ with $v\notin e$. Then $\beta_{n,n+k}(S/J_G)$ is the unique extremal Betti number of $S/J_G$.
				\end{proposition}
				\begin{proof}
					  As in the previous result, we prove that $n\leq \depth(S/J_G)$, $\reg(S/J_G)\leq k$ and $\beta_{n,n+k}(S/J_G)\neq 0$. Note that $G\setminus v=K_m\cup_e P_{k-1}$ with $r_1$ isolated vertices. By Proposition \ref{betti-decomposition}, $\reg(S/((x_v,y_v)+J_{G\setminus v}))=k-2$, and by \cite[Theorem 1.1]{her1}, $\pd(S/((x_v,y_v)+J_{G\setminus v}))=n-r_1\leq n-1$. Here, $G_v=K_m\cup_e C_{k-1}\cup_{\{v_2,v_k\}}K_{r_1+3}$ and $G_v\setminus v=K_m\cup_e C_{k-1}\cup_{\{v_2,v_k\}}K_{r_1+2}$. By virtue of Proposition \ref{betti-cliquesum2}, we have that $\beta_{n,n+k-1}(S/J_{G_v})$ and $\beta_{n+1,n+1+k-1}(S/((x_v,y_v)+J_{G_v\setminus v}))$ are the unique extremal Betti numbers of $S/J_{G_v}$ and $S/((x_v,y_v)+J_{G_v\setminus v})$ respectively. Then it follows from Lemmas \ref{depth-lemma}, \ref{regularity-lemma} and the short exact sequence \eqref{ohtani-ses} that $\depth(S/J_G)\geq n$ and $\reg(S/J_G)\leq k$. Now consider the long exact sequence \eqref{ohtani-tor} for $i=n+1$ and $j=k-1$, we get the isomorphism:
					  \begin{align*} \Tor_{n+1}^{S}\left(\frac{S}{(x_v,y_v)+J_{G_v\setminus v}},\K\right)_{n+1+k-1}\simeq \Tor_{n}^{S}\left( \frac{S}{J_G},\K\right)_{n+1+k-1}.
					  \end{align*} 
					  This implies that $\beta_{n,n+k}(S/J_G)\neq 0$, as desired. 
				\end{proof}
		\begin{proposition}\label{reg-whiskerk1}
			Let $k\geq 4$, $m\geq 2$ and $G=K_m\cup_e C_k\cup (\cup_{i=1}^{k} W^{r_i}(v_i))$ for $r_i\geq 0$. Let $A=\{v_i\in V(C_k): r_i\geq 1\}$. If $|A\cap e|=1$, $|A|=2$ and vertices of $A$ are not adjacent, then $\reg(S/J_G)=k$. Moreover, $S/J_G$ admits a unique extremal Betti number.
		\end{proposition}
		\begin{proof}
			Let $e=\{v,v_k\}$, $A=\{v,v_i\}$ such that $v$ and $v_i$ are non-adjacent. Then $3\leq i\leq k-1$ and $G=K_m\cup_e C_k\cup W^{r_1}(v)\cup W^{r_i}(v_i)$ for $r_1,r_i\geq 1$. It follows from Theorem \ref{depthn-girth} and Proposition \ref{reg-whiskerk-1_k} that $\depth(S/J_G)=n$ and $\reg(S/J_G)\leq k$ respectively. So we only need to show that $\beta_{n,n+k}(S/J_G)\neq 0$. Note that $G\setminus v$ is a block graph with $r_1$ isolated vertices. Thus, by \cite[Theorem 1.1]{her1} and \eqref{betti-product}, $\pd(S/((x_v,y_v)+J_{G\setminus v}))\leq n-1$. Also, it can be observed that $G_v=K_{m+r_1+1}\cup_{\{v_2,v_k\}} C_{k-1}\cup W^{r_i}(v_i)$ and $G_v\setminus v=K_{m+r_1}\cup_{\{v_2,v_k\}} C_{k-1}\cup W^{r_i}(v_i)$ with $v_i\notin \{v_2,v_k\}$. Then $G_v$ and $G_v\setminus v$ belong to the class of graphs considered in Proposition \ref{betti-cliquesum3}. Hence,  $\beta_{n,n+k-1}(S/J_{G_v})$ and $\beta_{n+1,n+1+k-1}(S/((x_v,y_v)+J_{G_v\setminus v}))$ are the unique extremal Betti numbers. Now it follows from the long exact sequence \eqref{ohtani-tor} for $i=n+1$ and $j=k-1$ that $\beta_{n,n+k}(S/J_G)\neq 0$. Hence, the assertion follows.
		\end{proof}	
		\begin{corollary}\label{reg-whiskerk}
			Let $k\geq 4$ and $G=C_k\cup (\cup_{i=1}^{k} W^{r_i}(v_i))$ for $r_i\geq 0$. Let $A=\{v_i\in V(C_k): r_i\geq 1\}$. If $|A|=2$ and vertices of $A$ are non-adjacent or $3\leq |A|\leq k-1$, then $\reg(S/J_G)=k$.
		\end{corollary}	
		\begin{proof}
		Let $v_j,v_l\in A$ such that $v_j$ and $v_{l}$ are non-adjacent. Set $G'=C_{k}\cup W^{r_j}(v_j)\cup W^{r_{l}}(v_{l})$. Then, clearly $G'$ is an induced subgraph of $G$. By considering $m=2$ in Proposition \ref{reg-whiskerk1}, we have $\reg(S_{G'}/J_{G'})=k$. Hence it follows from \cite[Corollary 2.2]{MM} and Proposition \ref{reg-whiskerk-1_k} that $\reg(S/J_G)=k$.
		\end{proof}
		We now combine the results from Theorem \ref{reg-whiskerk-1} and Corollaries \ref{reg-whisker}, \ref{reg-whiskerk+1}, \ref{reg-whiskerk} to get the following conclusion.
		\begin{corollary}\label{reg-corollary}
			Let $G=C_k\cup (\cup_{i=1}^{k} W^{r_i}(v_i))$, $r_i\geq 0$. Let $A=\{v_i\in V(C_k): r_i\geq 1\}$. If $A\neq \emptyset$, then $k-1\leq \reg(S/J_G)\leq k+1$. Moreover,
			\begin{enumerate}[(1)]
				\item $\reg(S/J_G)=k+1$ if and only if $A=V(C_k)$,
				\item $\reg(S/J_G)=k-1$ if and only if $|A|=1$ or $|A|=2$ and vertices of $A$ are adjacent,
				\item $\reg(S/J_G)=k$ if and only if $A$ contains at least two non-adjacent vertices and $A\subsetneq V(C_k)$.
			\end{enumerate}
		\end{corollary}
		Let $k\geq 4$ and $G=C_k\cup (\cup_{i=1}^{k} W^{r_i}(v_i))$, $r_i\geq 0$. If $\reg(S/J_G)=k+1$ or $\reg(S/J_G)=k-1$, then we proved that $S/J_G$ admits a unique extremal Betti number, see Corollary \ref{reg-whiskerk+1} and Theorem \ref{reg-whiskerk-1}. From now, we suppose $\reg(S/J_G)=k$. We show that $S/J_G$ does not always admit a unique extremal Betti number. In rest of the section, we study behavior of uniqueness of extremal Betti number for them. Let $A=\{v_i\in V(C_k): r_i\geq 1\}$. First, we consider the case when $G[A]$ is disconnected.
		\begin{proposition}\label{extremal-disconnected}
			Let $k\geq 4$ and $G=C_k\cup (\cup_{i=1}^{k} W^{r_i}(v_i))$ for $r_i\geq 0$. Let $A=\{v_i\in V(C_k): r_i\geq 1\}$. If $2\leq |A|\leq k-2$ and $G[A]$ is disconnected, then $\beta_{n,n+k}(S/J_G)$ is the unique extremal Betti number of $S/J_G$.
		\end{proposition}
		\begin{proof}
			Let $A=\{v_i\in V(C_k): r_i\geq 1\}$. Then $G=C_k\cup (\cup_{v_i\in A} W^{r_i}(v_i))$ for $r_i\geq 1$. If $|A|=2$, then we choose $e$ such that $A\cap e\neq \emptyset$. If $|A|\geq 3$, then choose $v_j\in A$ such that $G[A\setminus v_j]$ is disconnected and $e\cap A=\{v_j\}$. Let $H=K_m\cup_e G$ for $m\geq 2$. Set $n'=|V(H)|$. Then $n'=n+m-2$. We claim that $\beta_{n',n'+k}(S_H/J_H)\neq 0$. We prove this by induction on $|A|$. If $|A|=2$, then the assertion follows from Proposition \ref{reg-whiskerk1}. Suppose $|A|\geq 3$. Note that $H_{v_j}=K_{m+r_j+1}\cup_{\{v_{j-1},v_{j+1}\}}C_{k-1}\cup (\cup_{v_i\in A\setminus \{v_j\}} W^{r_i}(v_i))$ and $H_{v_j}\setminus v_j=K_{m+r_j}\cup_{\{v_{j-1},v_{j+1}\}}C_{k-1}\cup (\cup_{v_i\in A\setminus \{v_j\}} W^{r_i}(v_i))$. Since $G[A\setminus \{v_j\}]$ is disconnected with $|A\setminus \{v_j\}|\geq 2$, $H_{v_j}$ and $H_{v_j}\setminus v_j$ satisfy induction hypotheses. Therefore, $\beta_{n',n'+k-1}(S_H/J_{H_{v_j}})$ and $\beta_{n'+1,n'+1+k-1}(S_H/((x_{v_j},y_{v_j})+J_{H_{v_j}\setminus v_j}))$ are the unique extremal Betti numbers of $S_H/J_{H_{v_j}}$ and $S_H/((x_{v_j},y_{v_j})+J_{H_{v_j}\setminus v_j})$ respectively. By \cite[Theorem 1.1]{her1}, $\pd(S_H/((x_{v_j},y_{v_j})+J_{H\setminus v_j}))\leq n'-1$. Now consider the long exact sequence \eqref{ohtani-tor} for the pair $(H,v_j)$ to get that $\beta_{n',n'+k}(S_H/J_H)\neq 0$. Taking $m=2$, we get  $\beta_{n,n+k}(S/J_G)\neq 0$. By Corollaries \ref{depth-unicyclic} and \ref{reg-whiskerk}, we have $\depth(S/J_G)=n$ and $\reg(S/J_G)=k$. Hence, $\beta_{n,n+k}(S/J_G)$ is the unique extremal Betti number of $S/J_G$.
		\end{proof}
		From now, we suppose $G[A]$ is connected.
		\begin{proposition}\label{non-unique-extremal-whiskers}
			Let $k\geq 5$, $m\geq 3$ and $G=K_m\cup_e C_k\cup (\cup_{i=1}^{k} W^{r_i}(v_i))$ for $r_i\geq 0$. Let $A=\{v_i\in V(C_k): r_i\geq 1\}$. If $|A\cap e|=1$, $2\leq |A|\leq k-3$ and $G[A]$ is connected, then $\beta_{n,n+k-1}(S/J_G)$ is an extremal Betti number. In particular, if $k\geq 6$, $G=C_k\cup (\cup_{i=1}^{k} W^{r_i}(v_i))$ and $G[A]$ is connected with $3\leq |A|\leq k-3$, then $\beta_{n,n+k-1}(S/J_G)$ is an extremal Betti number, i.e., $S/J_G$ does not admit a unique extremal Betti number.
		\end{proposition}
		\begin{proof}
			Let $G=K_m\cup_e C_k\cup (\cup_{i=1}^{k} W^{r_i}(v_i))$ for $r_i\geq 0$. By Theorem \ref{depthn-girth}, we have either $\beta_{n,n+k-1}(S/J_G)$ or $\beta_{n,n+k}(S/J_G)$ is an extremal Betti number. So, it is enough to show that $\beta_{n,n+k}(S/J_G)=0$. We prove this by induction on $|A|$. Let $e=\{v,v_k\}$. Since $|A\cap e|=1$, assume that $v\in A$. Set $A=\{v,v_2,\dots,v_t\}$ for some $2\leq t\leq k-3$. Since, $G\setminus v$ is a disconnected block graph with $r_1+1$ components, by \cite[Theorem 1.1]{her1} and \eqref{betti-product}, $\pd(S/((x_v,y_v)+J_{G\setminus v}))=n-r_1\leq n-1$. Suppose $|A|=2$. Then it can be noted that $G_v=K_{m+r_1+1}\cup_{\{v_k,v_2\}} C_{k-1}\cup W^{r_2}(v_2)$ and $G_v\setminus v=K_{m+r_1}\cup_{\{v_k,v_2\}} C_{k-1}\cup W^{r_2}(v_2)$. Now it follows from the proof of Theorem \ref{reg-whisker-cliquesum} that $\reg(S/J_{G_v})=k-2=\reg(S/((x_v,y_v)+J_{G_v\setminus v}))$. Therefore, by Theorem \ref{depthn-girth}, $\beta_{n,n+k-2}(S/J_{G_v})$ and $\beta_{n+1,n+1+k-2}(S/((x_v,y_v)+J_{G_v\setminus v}))$ are extremal Betti numbers, and hence $\beta_{n,n+k}(S/J_{G_v})=0=\beta_{n+1,n+k}(S/((x_v,y_v)+J_{G_v\setminus v}))$. Thus, it follows from the long exact sequence \eqref{ohtani-tor} that $\beta_{n,n+k}(S/J_G)=0$. Now assume that $|A|\geq 3$. Then $G_v=K_{m+r_1+1}\cup_{\{v_k,v_2\}} C_{k-1}\cup (\cup_{i=2}^{t}W^{r_i}(v_i))$ and
			$G_v\setminus v=K_{m+r_1}\cup_{\{v_k,v_2\}} C_{k-1}\cup (\cup_{i=2}^{t}W^{r_i}(v_i))$. Clearly, $G_v$ and $G_v\setminus v$ satisfy induction hypotheses. Therefore, $\beta_{n,n+k-1}(S/J_{G_v})=0=\beta_{n+1,n+1+k-1}(S/((x_v,y_v)+J_{G_v\setminus v}))$.
			Hence, from the long exact sequence \eqref{ohtani-tor}, we get $\beta_{n,n+k}(S/J_G)=0$.
			
			Let $G=C_k\cup (\cup_{i=1}^{k} W^{r_i}(v_i))$ for $r_i\geq 0$. Set $A=\{v,v_2,\dots,v_t\}$ for some $3\leq t\leq k-3$. Then $G=C_k\cup (\cup_{i=1}^{t} W^{r_i}(v_i))$ for $r_i\geq 1$. Observe that $G_v=K_{r_1+3}\cup_{\{v_k,v_2\}} C_{k-1}\cup (\cup_{i=2}^{t}W^{r_i}(v_i))$ and
			$G_v\setminus v=K_{r_1+2}\cup_{\{v_k,v_2\}} C_{k-1}\cup (\cup_{i=2}^{t}W^{r_i}(v_i))$. Thus, by the above part and \eqref{betti-product}, $\beta_{n,n+k-1}(S/J_{G_v})=0=\beta_{n+1,n+1+k-1}(S/((x_v,y_v)+J_{G_v\setminus v}))$. Since, $G\setminus v$ is a forest with $r_1+1$ trees, by \cite[Theorem 1.1]{her1} and \eqref{betti-product}, $\pd(S/((x_v,y_v)+J_{G\setminus v}))=n-r_1\leq n-1$. Hence, it follows from the long exact sequence \eqref{ohtani-tor} that $\beta_{n,n+k}(S/J_G)=0$. By Corollary \ref{reg-whiskerk}, we have $\reg(S/J_G)=k$. Therefore, $\beta_{n,n+k-1}(S/J_G)$ is not the unique extremal Betti number of $S/J_G$.
		\end{proof}
		Now we consider the case when $|A|=k-2$ and $G[A]$ is connected. We assume that $A=\{v,v_2,\dots,v_{k-2}\}$. Then $G=C_k\cup (\cup_{i=1}^{k-2} W^{r_i}(v_i))$ for $r_i\geq 1$. Now, we investigate the uniqueness of extremal Betti number of $S/J_G$.
			\begin{proposition}\label{non-unique-extremal2}
				Let $k\geq 4$, $m\geq 3$ and $G=K_m\cup_{\{v,v_k\}}C_k\cup (\cup_{i=1}^{k-2} W^{r_i}(v_i))$ for $r_i\geq 1$. If $r_i\geq 2$ for all $1\leq i\leq k-3$, then $\beta_{n-1,n-1+k-1}(S/J_G)$ is an extremal Betti number. In particular, if $k\geq 5$, $G=C_k\cup (\cup_{i=1}^{k-2} W^{r_i}(v_i))$ with $r_i\geq 2$ for all $2\leq i\leq k-3$, then $\beta_{n-1,n-1+k-1}(S/J_G)$ is an extremal Betti number, i.e., $S/J_G$ does not admit a unique extremal Betti number.
			\end{proposition}
			\begin{proof}
			Let $G=K_m\cup_{\{v,v_k\}}C_k\cup (\cup_{i=1}^{k-2} W^{r_i}(v_i))$, $r_i\geq 1$ and suppose that $r_i\geq 2$ for all $1\leq i\leq k-3$. Then it follows from \cite[Theorem 1.1]{her1} and \eqref{betti-product} that $\pd(S/((x_v,y_v)+J_{G\setminus v}))=n-r_1\leq n-2$. Due to Theorem \ref{depthn+1-girth}, it is enough to show that $\beta_{n-1,n-1+k}(S/J_G)=0$. We proceed it by induction on $k$. Assume that $k=4$. Then  $G_v=K_{m+r_1+1}\cup_{\{v_2,v_k\}}C_3\cup W^{r_2}(v_2)$ and $G_v\setminus v=K_{m+r_1}\cup_{\{v_2,v_k\}}C_{3}\cup W^{r_2}(v_2)$. Thus, $G_v$ and $G_v\setminus v$ belong to the class of graphs considered in Theorem \ref{depth-girth3}. Hence, $\beta_{n-1,n-1+2}(S/J_{G_v})$ and $\beta_{n,n+2}(S/((x_v,y_v)+J_{G_v\setminus v}))$ are extremal Betti numbers. Therefore, $\beta_{n-1,n-1+4}(S/J_{G_v})=0=\beta_{n,n-1+4}(S/((x_v,y_v)+J_{G_v\setminus v}))$. Now it follows from the long exact sequence \eqref{ohtani-tor} for $i=n-1$ that $\beta_{n-1,n-1+4}(S/J_G)=0$. We assume that $k\geq 5$. Then $G_v=K_{m+r_1+1}\cup_{\{v_2,v_k\}}C_{k-1}\cup(\cup_{i=2}^{k-2} W^{r_i}(v_i))$ and $G_v\setminus v=K_{m+r_1}\cup_{\{v_2,v_k\}}C_{k-1}\cup (\cup_{i=2}^{k-2} W^{r_i}(v_i))$. Thus, by induction and \eqref{betti-product}, $\beta_{n-1,n-1+k-1}(S/J_{G_v})=0=\beta_{n,n+k-1}(S/((x_v,y_v)+J_{G_v\setminus v}))$. Hence, from the long exact sequence \eqref{ohtani-tor} for $i=n-1$, we get $\beta_{n-1,n-1+k}(S/J_G)=0$.
			
			Let $G=C_k\cup (\cup_{i=1}^{k-2} W^{r_i}(v_i))$, where $r_1,r_{k-2}\geq 1$ and $r_i\geq 2$ for all $2\leq i\leq k-3$. As in the above part, it is enough to show that $\beta_{n-1,n-1+k}(S/J_G)=0$. Note that $G\setminus v$ is the graph $P_{k-1}\cup (\cup_{i=2}^{k-2} W^{r_i}(v_i))$ with $r_1$ isolated vertices. So, $\iv(G\setminus v)=k-2$. Then by \cite[Theorem 8]{her2} and \eqref{betti-product} $\reg(S/((x_v,y_v)+J_{G\setminus v}))=k-1$, and hence $\beta_{n-1,n-1+k}(S/((x_v,y_v)+J_{G\setminus v}))=0$. Here, $G_v=K_{r_1+3}\cup_{\{v_2,v_k\}}C_{k-1}\cup (\cup_{i=2}^{k-2} W^{r_i}(v_i))$ and $G_v\setminus v=K_{r_1+2}\cup_{\{v_2,v_k\}}C_{k-1}\cup (\cup_{i=2}^{k-2} W^{r_i}(v_i))$. Therefore, by the above part and \eqref{betti-product}, $\beta_{n-1,n-1+k-1}(S/J_{G_v})=0=\beta_{n,n+k-1}(S/((x_v,y_v)+J_{G_v\setminus v}))$. Now it follows from the long exact sequence \eqref{ohtani-tor} that $\beta_{n-1,n-1+k}(S/J_G)=0$, as required.
			\end{proof}
			\begin{proposition}\label{unique-extremal-whisker-lemma}
				Let $k\geq 4$ and $m\geq 3$. Let $G=K_m\cup_{\{v,v_k\}}C_k\cup (\cup_{i=1}^{k-2} W^{r_i}(v_i))$ for $r_i\geq 1$. If $r_i=1$ for some $1\leq i\leq k-3$, then $\beta_{n-1,n-1+k}(S/J_G)$ is an extremal Betti number of $S/J_G$. In particular, if $k\geq 5$, $G=C_k\cup (\cup_{i=1}^{k-2} W^{r_i}(v_i))$ and $r_i=1$ for some $2\leq i\leq k-3$, then $\beta_{n-1,n-1+k}(S/J_G)$ is the unique extremal Betti number of $S/J_G$.
			\end{proposition}
			\begin{proof}
				Due to Theorem \ref{depthn+1-girth}, it is enough to show that  $\beta_{n-1,n-1+k}(S/J_G)\neq 0$. To prove this we proceed by induction on $k$. Assume that $k=4$. Then $G=K_m\cup_{\{v,v_k\}}C_4\cup W^{r_1}(v)\cup W^{r_2}(v_2)$ for $r_1=1$ and $r_2\geq 1$. In this case, $G\setminus v$ is the graph $K_{m-1}\cup_{v_k} P_3\cup W^{r_2}(v_2)$ with one isolated vertex. Therefore, by \cite[Theorem 1.1]{her1} and \eqref{betti-product}, $\pd(S/((x_v,y_v)+J_{G\setminus v}))=n-1$, and hence by \cite[Theorem 8]{her2} and \eqref{betti-product}, $\beta_{n-1,n-1+4}(S/((x_v,y_v)+J_{G\setminus v}))$ is an extremal Betti number. Note that $G_v=K_{m+2}\cup _{\{v_2,v_k\}}C_3\cup W^{r_2}(v_2)$ and $G_v\setminus v=K_{m+1}\cup _{\{v_2,v_k\}}C_3\cup W^{r_2}(v_2)$. By Proposition \ref{reg-whiskerk-1_k}, $\reg(S/J_{G_v})\leq 3$ and $\reg(S/((x_v,y_v)+J_{G_v\setminus v}))\leq 3$. Therefore, $\beta_{n-1,n-1+j}(S/J_{G_v})=0=\beta_{n-1,n-1+j}(S/((x_v,y_v)+J_{G_v\setminus v}))$ for $j\geq 4$. Hence it follows from the long exact sequence \eqref{ohtani-tor} that $\beta_{n-1,n-1+4}(S/J_G)=\beta_{n-1,n-1+4}(S/((x_v,y_v)+J_{G\setminus v}))\neq 0$. 
				
				Now, we assume that $k\geq 5$. Let $r_i=1$ for some $1\leq i\leq k-3$. Then $G_v=K_{m+r_1+1}\cup_{\{v_2,v_k\}}C_{k-1}\cup (\cup_{i=2}^{k-2} W^{r_i}(v_i))$ and $G_v\setminus v=K_{m+r_1}\cup_{\{v_2,v_k\}}C_{k-1}\cup (\cup_{i=2}^{k-2} W^{r_i}(v_i))$.
				\vskip 1mm
				\noindent
				\textbf{Case 1:} Let $r_1=1$ and $r_i\geq 2$ for all $2\leq i\leq k-3$. Then $G_v=K_{m+2}\cup_{\{v_2,v_k\}}C_{k-1}\cup (\cup_{i=2}^{k-2} W^{r_i}(v_i))$ and $G_v\setminus v=K_{m+1}\cup_{\{v_2,v_k\}}C_{k-1}\cup (\cup_{i=2}^{k-2} W^{r_i}(v_i))$. By virtue of Proposition \ref{reg-whiskerk-1_k}, we have $\reg(S/J_{G_v})\leq k-1$ and $\reg(S/((x_v,y_v)+J_{G_v\setminus v}))\leq k-1$. Hence, $\beta_{n-1,n-1+j}(S/J_{G_v})=0=\beta_{n-1,n-1+j}(S/((x_v,y_v)+J_{G_v\setminus v}))$ for $j\geq k$. In this case, $G\setminus v$ is the graph $K_{m-1}\cup_{v_k} P_{k-1}\cup (\cup_{i=2}^{k-2} W^{r_i}(v_i))$ with one isolated vertex. Therefore, by \cite[Theorem 1.1]{her1} and \eqref{betti-product}, $\pd(S/((x_v,y_v)+J_{G\setminus v}))=n-1$, and hence by \cite[Theorem 8]{her2} and \eqref{betti-product}, $\beta_{n-1,n-1+k}(S/((x_v,y_v)+J_{G\setminus v}))$ is an extremal Betti number. Therefore, from the long exact sequence \eqref{ohtani-tor}, we get that $\beta_{n-1,n-1+k}(S/J_G)=\beta_{n-1,n-1+k}(S/((x_v,y_v)+J_{G\setminus v}))\neq 0$.
				\vskip 1mm
				\noindent
				\textbf{Case 2:} Let $r_i=1$ for some $2\leq i\leq k-3$. By \cite[Theorem 1.1]{her1} and \eqref{betti-product}, $\pd(S/((x_v,y_v)+J_{G\setminus v}))=n-r_1\leq n-1$. In this case, notice that $G_v$ and $G_v\setminus v$ satisfy induction hypotheses. Therefore, $\beta_{n-1,n-1+k-1}(S/J_{G_v})$ and $\beta_{n,n+k-1}(S/((x_v,y_v)+J_{G_v\setminus v}))$ are extremal Betti numbers. Then it follows from the long exact sequence \eqref{ohtani-tor} for $i=n-1$ and $j=k$ that $\beta_{n-1,n-1+k}(S/J_G)\neq 0$.
				
				As in the above part, it is enough to show that $\beta_{n-1,n-1+k}(S/J_G)\neq 0$. Note that $G\setminus v$ is the graph $P_{k-1}\cup (\cup_{i=2}^{k-2} W^{r_i}(v_i))$ with $r_1$ isolated vertices. Therefore, by \cite[Theorem 1.1]{her1}, $p=\pd(S/((x_v,y_v)+J_{G\setminus v}))=n-r_2\leq n-1$. Observe that $G_v=K_{r_1+3}\cup_{\{v_2,v_k\}}C_{k-1}\cup (\cup_{i=2}^{k-2} W^{r_i}(v_i))$ and $G_v\setminus v=K_{r_1+2}\cup_{\{v_2,v_k\}}C_{k-1}\cup (\cup_{i=2}^{k-2} W^{r_i}(v_i))$. Therefore, by the above part and \eqref{betti-product}, $\beta_{n-1,n-1+k-1}(S/J_{G_v})$ and $\beta_{n,n+k-1}(S/((x_v,y_v)+J_{G_v\setminus v}))$ are extremal Betti numbers. Hence the assertion follows from the long exact sequence \eqref{ohtani-tor} for $i=n-1$ and $j=k$.
				\end{proof}
				Now we are left with the case that $|A|=k-1$. In this case, we prove that $S/J_G$ admits a unique extremal Betti number.
				\begin{proposition}\label{betti-whiskerk2}
					Let $H=K_m\cup_e C_k$ for $k\geq 3$, $m\geq 2$. Let $G=H\cup (\cup_{i=1}^{k} W^{r_i}(v_i))$ for $r_i\geq 0$. Let $A=\{v_i\in V(C_k): r_i\geq 1\}$. If $|A|=k-1$ and $e\nsubseteq A$, then $\beta_{n-1,n-1+k}(S/J_G)$ is an extremal Betti number. In particular, if $k\geq 3$, $G=C_k\cup (\cup_{i=1}^{k} W^{r_i}(v_i))$ with $|A|=k-1$, then $S/J_G$ admits a unique extremal Betti number.
				\end{proposition}
				\begin{proof}
					As in the previous result, it is enough to show that $\beta_{n-1,n-1+k}(S/J_G)\neq 0$. Let $e=\{v,v_2\}$. We may assume that $r_2=0$. Then $G=K_m\cup_e C_k\cup (\cup_{i=1,i\neq 2}^{k} W^{r_i}(v_i))$ for $r_i\geq 1$. We prove the first part by induction on $k$. If $k=3$ and $m=2$, then the result follows from \cite[Theorem 8]{her2}. If $k=3$ and $m\geq 3$, then the result follows from Theorem \ref{depth-girth3}. Now assume that $k\geq 4$. Note that $G_v=K_{m+r_1+1}\cup_{\{v_2,v_k\}}C_{k-1}\cup(\cup_{i=3}^{k} W^{r_i}(v_i))$, $r_i\geq 1$ and $G_v\setminus v=K_{m+r_1}\cup_{\{v_2,v_k\}}C_{k-1}\cup (\cup_{i=3}^{k} W^{r_i}(v_i))$, $r_i\geq 1$. Thus, by induction and \eqref{betti-product}, $\beta_{n-1,n-1+k-1}(S/J_{G_v})$ and $\beta_{n,n+k-1}(S/((x_v,y_v)+J_{G_v\setminus v}))$ are extremal Betti numbers. Also, by \cite[Theorem 1.1]{her1}, $\pd(S/((x_v,y_v)+J_{G\setminus v}))=n-r_1\leq n-1$. Hence it follows from the long exact sequence \eqref{ohtani-tor} for $i=n-1$ and $j=k$ that $\beta_{n-1,n-1+k}(S/J_G)\neq 0$. Taking $m=2$, we get the second assertion.
				\end{proof}
			We now conclude the following result for the behavior of uniqueness of extremal Betti number for cycles with whiskers graphs.
			\begin{corollary}
			Let $G=C_k\cup (\cup_{i=1}^{k} W^{r_i}(v_i))$, $r_i\geq 0$ with $\reg(S/J_G)=k$. Let $A=\{v_i\in V(C_k): r_i\geq 1\}$.
			\begin{enumerate}[(1)]
				\item If $2\leq |A|\leq k-2$ and $G[A]$ is disconnected, then $S/J_G$ admits a unique extremal Betti number.
				\item Suppose $G[A]$ is connected:
				\begin{enumerate}[(a)]
					\item If $3\leq |A|\leq k-3$, then $S/J_G$ does not admit a unique extremal Betti number.
					\item Suppose $|A|=k-2$, $A=\{v_1,\dots,v_{k-2}\}$. If $r_i\geq 2$ for all $2\leq i\leq k-3$, then $S/J_G$ does not admit a unique extremal Betti number.
					\item Suppose $|A|=k-2$, $A=\{v_1,\dots,v_{k-2}\}$. If $r_i=1$ for some $2\leq i\leq k-3$, then $S/J_G$ admits a unique extremal Betti number
					\item If $|A|=k-1$, then $S/J_G$ admits a unique extremal Betti number.
				\end{enumerate}
			\end{enumerate}	
			\end{corollary}
			 
		To get a better insight into our results, let us look at some of the following examples:

			\begin{minipage}{\linewidth}
				\begin{minipage}{.22\linewidth}
					\begin{figure}[H]
						\begin{tikzpicture}[scale=.8]
						\draw (4.04,-2.14)-- (3.5,-2.66);
						\draw (4.64,-2.68)-- (4.04,-2.14);
						\draw (3.5,-2.66)-- (3.5,-3.43);
						\draw (2.85,-3.40)-- (3.5,-3.43);
						\draw (3.5,-3.43)-- (4.66,-3.43);
						\draw (4.66,-3.43)-- (4.64,-2.68);
						\draw (4.66,-3.43)-- (5.50, -3.40);
						\draw (5.27,-2.58)-- (4.64,-2.68);
						\draw (2.85,-2.56)-- (3.5,-2.66);
						\draw (3.7,-1.7)-- (4.04,-2.14);
						\draw (4.46,-1.62)-- (4.04,-2.14);
						\begin{scriptsize}
						\node at (4.1, -3) {n=11};
						\node at (4.1, -3.7) {$G_1$};
						\fill (4.04,-2.14) circle (1.5pt);
						\fill (5.50,-3.40) circle (1.5pt);
						\fill (2.85,-3.40) circle (1.5pt);
						\fill (3.5,-2.66) circle (1.5pt);
						\fill (4.64,-2.68) circle (1.5pt);
						\fill (3.5,-3.43) circle (1.5pt);
						\fill (4.66,-3.43) circle (1.5pt);
						\fill (2.85,-2.56) circle (1.5pt);
						\fill (5.27,-2.58) circle (1.5pt);
						\fill (3.7,-1.7) circle (1.5pt);
						\fill (4.46,-1.62) circle (1.5pt);
						\end{scriptsize}
						\end{tikzpicture}
					\end{figure}
				\end{minipage}
				\begin{minipage}{.22\linewidth}
					\begin{figure}[H]
						\begin{tikzpicture}[scale=.8]
						\draw (4.04,-2.14)-- (3.5,-2.66);
						\draw (4.64,-2.68)-- (4.04,-2.14);
						\draw (3.5,-2.66)-- (3.5,-3.43);
						\draw (3.5,-3.43)-- (4.66,-3.43);
						\draw (4.66,-3.43)-- (4.64,-2.68);
						\draw (2.85,-2.56)-- (3.5,-2.66);
						\draw (3.7,-1.7)-- (4.04,-2.14);
						\draw (4.46,-1.62)-- (4.04,-2.14);
						\begin{scriptsize}
						\node at (4.1, -3) {n=8};
						\node at (4.1, -3.7) {$G_2$};
						\fill (4.04,-2.14) circle (1.5pt);
						\fill (3.5,-2.66) circle (1.5pt);
						\fill (4.64,-2.68) circle (1.5pt);
						\fill (3.5,-3.43) circle (1.5pt);
						\fill (4.66,-3.43) circle (1.5pt);
						\fill (2.85,-2.56) circle (1.5pt);
						\fill (3.7,-1.7) circle (1.5pt);
						\fill (4.46,-1.62) circle (1.5pt);
						\end{scriptsize}
						\end{tikzpicture}
					\end{figure}
				\end{minipage}
				\begin{minipage}{.22\linewidth}
					\begin{figure}[H]
						\begin{tikzpicture}[scale=.8]
						\draw (4.04,-2.14)-- (3.5,-2.66);
						\draw (4.64,-2.68)-- (4.04,-2.14);
						\draw (3.5,-2.66)-- (3.5,-3.43);
						\draw (3.5,-3.43)-- (4.66,-3.43);
						\draw (4.66,-3.43)-- (4.64,-2.68);
						\draw (4.66,-3.43)-- (5.50, -3.40);
						\draw (2.85,-2.56)-- (3.5,-2.66);
						\draw (3.7,-1.7)-- (4.04,-2.14);
						\draw (4.46,-1.62)-- (4.04,-2.14);
						\begin{scriptsize}
						\node at (4.1, -3) {n=9};
						\node at (4.1, -3.7) {$G_3$};
						\fill (4.04,-2.14) circle (1.5pt);
						\fill (5.50,-3.40) circle (1.5pt);
						\fill (3.5,-2.66) circle (1.5pt);
						\fill (4.64,-2.68) circle (1.5pt);
						\fill (3.5,-3.43) circle (1.5pt);
						\fill (4.66,-3.43) circle (1.5pt);
						\fill (2.85,-2.56) circle (1.5pt);
						\fill (3.7,-1.7) circle (1.5pt);
						\fill (4.46,-1.62) circle (1.5pt);
						\end{scriptsize}
						\end{tikzpicture}
					\end{figure}
				\end{minipage}
				\begin{minipage}{.22\linewidth}
					\begin{figure}[H]
						\begin{tikzpicture}[scale=.8]
						\draw (4.04,-2.14)-- (3.5,-2.66);
						\draw (4.64,-2.68)-- (4.04,-2.14);
						\draw (3.5,-2.66)-- (3.5,-3.43);
						\draw (3.5,-3.43)-- (4.66,-3.43);
						\draw (4.66,-3.43)-- (4.64,-2.68);
						\draw (5.27,-2.58)-- (4.64,-2.68);
						\draw (2.85,-2.56)-- (3.5,-2.66);
						\draw (3.7,-1.7)-- (4.04,-2.14);
						\draw (4.46,-1.62)-- (4.04,-2.14);
						\begin{scriptsize}
						\node at (4.1, -3) {n=9};
						\node at (4.1, -3.7) {$G_4$};
						\fill (4.04,-2.14) circle (1.5pt);
						\fill (3.5,-2.66) circle (1.5pt);
						\fill (4.64,-2.68) circle (1.5pt);
						\fill (3.5,-3.43) circle (1.5pt);
						\fill (4.66,-3.43) circle (1.5pt);
						\fill (2.85,-2.56) circle (1.5pt);
						\fill (5.27,-2.58) circle (1.5pt);
						\fill (3.7,-1.7) circle (1.5pt);
						\fill (4.46,-1.62) circle (1.5pt);
						\end{scriptsize}
						\end{tikzpicture}
					\end{figure}
				\end{minipage}				
			\end{minipage}
		    \vskip 3mm
		    \noindent		   
			By Corollary \ref{reg-whiskerk+1}, Theorem \ref{reg-whiskerk-1} and Corollary \ref{reg-whiskerk}, $\reg(S_{G_1}/J_{G_1})=6$, $\reg(S_{G_2}/J_{G_2})=4$, $\reg(S_{G_3}/J_{G_3})=5$ and $\reg(S_{G_4}/J_{G_4})=5$. Also, we get that $\beta_{10,16}(S_{G_1}/J_{G_1})$ and $\beta_{8,12}(S_{G_2}/J_{G_2})$  are the unique extremal Betti numbers of $S_{G_1}/J_{G_1}$ and $S_{G_2}/J_{G_2}$ . By Proposition \ref{extremal-disconnected}, $\beta_{9,14}(S_{G_3}/J_{G_3})$ is the unique extremal Betti number of $S_{G_3}/J_{G_3}$. By Proposition \ref{non-unique-extremal2}, $\beta_{8,12}(S_{G_4}/J_{G_4})$ is an extremal Betti number of $S_{G_4}/J_{G_4}$, i.e., $S_{G_4}/J_{G_4}$ does not admit a unique extremal Betti number.

			It will be interesting to obtain an answer to:
			\begin{question}
				Characterize unicyclic graphs $G$ such that $S/J_G$ admits a unique extremal Betti number.
			\end{question}

\bibliographystyle{plain}
\bibliography{Reference}

\end{document}